\documentclass[12pt,twoside]{amsart}
\usepackage{amssymb}
\usepackage{amscd}
\usepackage{color}
\usepackage{comment}			
\usepackage{url}						
\usepackage{mathtools}			
\usepackage{hyperref}

\title[Relative augmented base loci]
{Relative augmented base loci on noetherian schemes} 
\author{Yusuke Ushiro} 
\address{Graduate School of Mathematical Sciences, The University of Tokyo, 3-8-1 Komaba, Meguro-ku, Tokyo 153-8914, JAPAN} 
\email{ushiro-yusuke8100@g.ecc.u-tokyo.ac.jp}

\newcommand{\red}[0]{{\operatorname{red}}}
\newcommand{\Ker}[0]{{\operatorname{Ker}}}
\newcommand{\Coker}[0]{{\operatorname{Coker}}}

\newcommand{\Spec}[0]{{\operatorname{Spec}}}

\newcommand{\Bs}[0]{{\operatorname{Bs}}}
\newcommand{\Supp}[0]{{\operatorname{Supp}}}

\newcommand{\Ass}[0]{{\operatorname{Ass}}}
\newcommand{\id}[0]{{\operatorname{id}}}
\newcommand{\transdeg}[0]{{\operatorname{tr}\hspace{-3pt}.\hspace{-2pt}\operatorname{deg}}}
\newcommand{\glen}[0]{{\operatorname{ge}\hspace{-2pt}.\hspace{-2pt}\operatorname{len}}}
\newcommand{\E}[0]{{\mathbf{E}}}
\newcommand{\SB}[0]{{\mathbf{SB}}}
\newcommand{\B}[0]{{\mathbf{B}}}
\renewcommand{\Im}[0]{{\operatorname{Im}}}

\newcommand{\sequence}[3]{{#1_1 #2 #1_2 #2 \dots #2 #1_{#3}}}

\newtheorem{thm}{Theorem}[section]
\newtheorem{lem}[thm]{Lemma}
\newtheorem{cor}[thm]{Corollary}
\newtheorem{prop}[thm]{Proposition}

\newtheorem{claim}[thm]{Claim}

\theoremstyle{definition}

\newtheorem{dfn}[thm]{Definition}
\newtheorem{dfnprop}[thm]{Definition-Proposition}
\newtheorem{rem}[thm]{Remark}
\newtheorem*{ack}{Acknowledgments}

\newcommand{\ME}{\mathcal{E}}
\newcommand{\MF}{\mathcal{F}}
\newcommand{\MG}{\mathcal{G}}
\newcommand{\MI}{\mathcal{I}}
\newcommand{\MN}{\mathcal{N}}
\newcommand{\MJ}{\mathcal{J}}
\newcommand{\MK}{\mathcal{K}}
\newcommand{\MO}{\mathcal{O}}

\newcommand{\Mp}{\mathfrak{p}}
\newcommand{\Mq}{\mathfrak{q}}
\newcommand{\Mm}{\mathfrak{m}}
\newcommand{\R}{\mathbb{R}}
\newcommand{\Q}{\mathbb{Q}}
\newcommand{\Z}{\mathbb{Z}}
\newcommand{\Fp}{\mathbb{F}_p}
\newcommand{\C}{\mathbb{C}}
\begin{document}

\begin{abstract}
	Let $f:X \to S$ be a projective morphism of noetherian schemes
	and let $L$ be an invertible sheaf on $X$.
	We show that 
	the relative augmented base locus of $L$
	coincides with the relative exceptional locus of $L$.
	We also prove a semi-ampleness criterion in terms of the exceptional locus generalizing the result of Keel.
\end{abstract}

\maketitle

\tableofcontents

\setcounter{section}{0}

\section{Introduction}

Let $f:X \to S$ be a projective morphism of noetherian schemes,
let $L$ be an invertible sheaf on $X$,
and let $A$ be an $f$-ample invertible sheaf on $X$.
Then the $f$-stable base locus $\SB_f(L)$ of $L$ is defined as 
\[
\SB_f(L) 
\coloneqq \bigcap_{m \in \Z_{>0}} \Supp \  \Coker(f^* f_* L^{\otimes m} \to L^{\otimes m}),
\]
and the $f$-augmented base locus $\B_{+,f}(L)$ of $L$ is defined as
\[
\B_{+,f}(L) 
\coloneqq \bigcap_{m \in \Z_{>0}} \SB_f(L^{\otimes m} \otimes_{\MO_X} A^{-1})
\]
(cf. Definition~\ref{def:base-locus}).
We define the $f$-exceptional locus $\E_f(L)$ of $L$ 
as the set-theoretic union of all the integral closed subschemes $V$ of $X$
such that $L|_V$ are not $f|_V$-big (cf. Definition~\ref{def:exceptional-locus}).

The main results of this paper are the following two theorems,
Theorem~\ref{intro:Nakamaye} and Theorem~\ref{intro:Keel}.
These are proved in Section~\ref{sec:K}.

\begin{thm}[Theorem~\ref{K-main1}] \label{intro:Nakamaye}
	Let $f:X \to S$ be a projective morphism 
	of noetherian schemes.
	Let $L$ be an $f$-nef invertible sheaf on $X$.
	Then, we have $\B_{+,f}(L)=\E_f(L)$. 
\end{thm}

Nakamaye first showed the theorem
for the case when $X$ is a smooth variety
and $S = \Spec\,k$, 
where $k$ is an algebraically closed field of characteristic $0$ 
(cf. \cite{Nak00}).
Ein--Lazarsfeld--Musta\c{t}\u{a}--Nakamaye--Popa showed analogous results 
for the case when $X$ is a smooth complex variety,
$S=\Spec\, \C$
and $L$ is an $\R$-divisor (\cite{ELMNP}).
Cacciola--Lopez showed the theorem 
for the case when $X$ is a normal variety,
the dimension of non-lc locus of $X$ is less than or equal to $1$, 
$S = \Spec\, \C$
and $L$ is a $\Q$-Cartier $\Q$-divisor (cf. \cite{CL14}).
Boucksom--Cacciola--Lopez proved related results concerning the restricted volume 
for the case when $X$ is a normal variety and $S = \Spec\, k$, where $k$ is an algebraically closed field (\cite{BCL14}).
%
%
Cascini--M\textsuperscript{c}Kernan--Musta\c{t}\u{a} proved the theorem 
for the case when $S = \Spec\,k$, 
where $k$ is an algebraically closed field of positive characteristic
 (cf. \cite{CMM}).
By using Fujita's vanishing theorem,
Birkar proved the theorem 
for the case when $S = \Spec\,k$ with an arbitrary field $k$
and $L$ is an $\R$-Cartier $\R$-divisor
 (cf. \cite{Bir17}).
Stigant mainly showed 
the theorem 
for the case when $S$ is excellent 
and $\B_{+,g}(L|_{X_\Q})=\E_g(L|_{X_\Q})$
where $g:X_\Q \to S_\Q$ is the induced morphism
(cf. \cite{Sti21}).

\begin{thm}[Theorem~\ref{K-main2}] \label{intro:Keel}
	Let $f:X \to S$ be a projective morphism 
	of noetherian schemes 
	and let $L$ be an $f$-nef invertible sheaf on $X$.
	Then, there exists a closed subscheme $F$ of $X$
	such that
	$\Supp(F) = \E_f(L)$
	and 
	$\SB_f(L) = \SB_{f|_F}(L|_F)$ hold.
\end{thm}

%
%
Keel first proved the theorem
for the case when $S = \Spec\, k$ with a field $k$ of positive characteristic
and $\SB_f(L) = \emptyset$ (cf. \cite{Kee99}).
Cascini--M\textsuperscript{c}Kernan--Musta\c{t}\u{a} simplified the proof 
for the case when $k$ is an algebraically closed field of positive characteristic (cf. \cite{CMM}).
Birkar proved 
the theorem
for the case when $S = \Spec\, k$
with an arbitrary field $k$,
$L$ is an $\Q$-Cartier $\Q$-divisor 
and $\SB_f(L) = \emptyset$
(cf. \cite{Bir17}).
Cascini--Tanaka showed the theorem
for the case when $S$ is an $\Fp$-scheme
(cf. \cite[Subsection~2.5]{CT}).
Although their definition of $\E_f(L)$ is slightly different from our definition, we shall show that these two definitions are equivalent (cf. Remark~\ref{rem:exceptional-locus}).
Stigant showed the theorem
for the case when $S$ is excellent, $L|_{X_\Q}$ is semi-ample and $F$ is reduced.
(cf. \cite{Sti21}).

In Section~\ref{sec:App},
we mainly show the following result.

\begin{thm}[Corollary~\ref{App-Kodaira's lemma}]\label{intro:App-Kodaira's lemma}
	Let $f:X \to S$ be a projective morphism of noetherian schemes,
	with $X$ reduced and $S$ affine.
	Let $L$ be an invertible sheaf on $X$, and
	let $A$ be an $f$-ample invertible sheaf on $X$.
	If $L|_V$ is $f|_V$-big
	for any irreducible component $V$ of $X$,
	then for sufficiently large $m \in \Z_{>0}$,
	there is an effective Cartier divisor $D$ on $X$ 
	such that $L^{\otimes m} \simeq A \otimes_{\MO_X} \MO_X(D)$.
\end{thm}

\subsection{Sketch of the proofs}

To prove Theorem~\ref{intro:Nakamaye} and Theorem~\ref{intro:Keel},
we basically follow the strategies of \cite{CMM} and \cite{Bir17}.
In place of the dimension over fields,
we use the generic length over base schemes.

\begin{dfn}[Definition~\ref{def:glen}]
	\begin{enumerate}
		\item
			Let $A$ be a noetherian ring which has only one minimal prime ideal $\Mp$
			and let $M$ be a finitely generated $A$-module.
			Then, 
			we define the {\em generic $A$-length} $\glen_A\, M$ of $M$ as
			the length of $A_\Mp$-module $M_\Mp$.
		\item
			Let $X$ be an irreducible noetherian scheme
			and let $\MF$ be a coherent sheaf on $X$.
			Let $\xi \in X$ be the generic point of $X_{\red}$.
			Then, 
			we define the {\em generic $X$-length} $\glen_X\, \MF$ of $\MF$ as
			the length of $\MO_{X, \xi}$-module $\MF_\xi$.
	\end{enumerate}
\end{dfn}

Note that 
the generic $X$-length is not defined
when $X$ is not irreducible.
%
We show the following two inequalities: Lemma~\ref{intro:lem:K-1} and Lemma~\ref{intro:lem:K-3}.

\begin{lem}[Lemma~\ref{lem:K-1}]\label{intro:lem:K-1}
	Let $f:X \to S$ be a projective morphism of noetherian schemes,
	with $S$ irreducible.
	Let $d \coloneqq \dim(X/S)$.
	Let $\MF$ be a coherent sheaf on $X$, and
	let $L$ be an $f$-nef invertible sheaf on $X$.
	Then, for any $j \in \Z_{\ge0}$, 
	there is a positive real number $C \in \R_{>0}$
	such that for sufficiently large $m \in \Z_{>0}$,
	\[
	\glen_S\,R^j f_* (\MF \otimes_{\MO_X} L ^{\otimes m})
	\le C m^{d-j}
	\]
	holds.
\end{lem}

\begin{proof}[Sketch of the proof of Lemma~\ref{intro:lem:K-1}]
	Let $\MN$ be the nilradical ideal sheaf of $X$.
	By using the exact sequence 
	\[
	0
	\to \MN^{i+1}
	\to \MN^i
	\to \MN^i / \MN^{i+1}
	\to 0,
	\]
	we reduce the proof to the case 
	where $X$ is reduced.
	When $X$ is reduced,
	by taking the generic fibre of $f(X)$,
	we can reduce the proof to the case 
	where $S = \Spec\,k$
	with a field $k$,
	which is shown by Birkar (cf. \cite{Bir17}).
\end{proof}

\begin{lem}[Lemma~\ref{lem:K-3}]\label{intro:lem:K-3}
	Let $f:X \to S$ be a projective surjective morphism 
	of irreducible noetherian schemes.
	Let $d \coloneqq \dim(X/S)$.
	Let $\ME$ be a coherent locally free sheaf on $X$, and
	let $L$ be an $f$-nef invertible sheaf on $X$.
	If $L$ is $f$-big,
	there is a positive real number $C \in \R_{>0}$
	such that for sufficiently large $m \in \Z_{>0}$,
	\[
	\glen_S\, f_* (\ME \otimes_{\MO_X} L ^{\otimes m})
	\ge C m^{d}
	\]
	holds.
\end{lem}

The proof is similar to the one of Lemma~\ref{intro:lem:K-1}.
The key lemma is the following.


\begin{lem}[Lemma~\ref{lem:K-4}] \label{intro:lem:K-4}
	Let $f:X \to S$ be a projective morphism of noetherian schemes,
	with $S$ affine.
	Let $L$ be an $f$-nef invertible sheaf on $X$, and
	let $A$ be an $f$-ample invertible sheaf on $X$.
	Assume that 
	$L$ is $f$-weakly big.
	Then, for sufficiently large $m \in \Z_{>0}$,
	there is a section $s \in H^0(X,L^{\otimes m}\otimes_{\MO_X} A^{-1})$
	such that 
	$s|_{X_\red} \neq 0$ holds.
\end{lem}

By using Lemma~\ref{intro:lem:K-1} and Lemma~\ref{intro:lem:K-3} 
we show the Lemma~\ref{intro:lem:K-4}.

\begin{proof}[Sketch of the proof of Lemma~\ref{intro:lem:K-4}]
	We show the following simple two cases:
	\begin{enumerate}
		\renewcommand{\labelenumi}{(\roman{enumi})}
		\item 
			$X$ is reduced;
		\item 
			$X$ is irreducible.
	\end{enumerate}
	We can show the general case
	by the combination of this two cases.
	
	(i)
	There is an irreducible component $X_1$ of $X$ 
	such that $L|_{X_1}$ is $f|_{X_1}$-big.
	Let $X'$ be the union of the other irreducible components.
	Let $T$ be the scheme-theoretic intersection of $X_1$ and $X'$.
	We may assume that $S$ and $f(X_1)$ are affine
	and write $f(X_1) = \Spec\, R$ with a noetherian ring $R$.
	Let $d=\dim(X_1 / f(X_1))$ and
	let $M_m \coloneqq L^{\otimes m} \otimes_{\MO_X} A^{-1}$.
	By Lemma~\ref{intro:lem:K-3},
	there is a positive real number $C_1 \in \R_{>0}$ 
	such that for sufficiently large $m \in \Z_{>0}$,
	\[
	\glen_R\, H^0(X_1,M_m|_{X_1}) \ge C_1 m^d
	\]
	holds.
	By Lemma~\ref{intro:lem:K-1},
	there is a positive real number $C_2 \in \R_{>0}$ 
	such that for sufficiently large $m \in \Z_{>0}$,
	\[
	\glen_R\, H^0(T,M_m|_T) \le C_2 m^{d-1}
	\]
	holds.
	By the exact sequence 
	\[
	H^0(X,M_m)
	\to H^0(X_1,M_m|_{X_1}) \oplus H^0(X',M_m|_{X'})
	\to H^0(T,M_m|_T),
	\]
	we can show that
	the natural morphism 
	$H^0(X,M_m) \to H^0(X_1,M_m|_{X_1})$ 
	is not zero
	for sufficiently large $m$.
	
	(ii)
	We may assume that $S$ and $f(X)$ are affine
	and write $f(X) = \Spec\, R$ with a noetherian ring $R$.
	Let $\MN$ be the nilradical ideal sheaf of $X$.
	Let $d=\dim(X / f(X))$ and
	let $M_m \coloneqq L^{\otimes m} \otimes_{\MO_X} A^{-1}$.
	By Lemma~\ref{intro:lem:K-3},
	there is a positive real number $C_1 \in \R_{>0}$ 
	such that for sufficiently large $m \in \Z_{>0}$,
	\[
	\glen_R\, H^0(X_\red,M_m|_{X_\red}) \ge C_1 m^d
	\]
	holds.
	By Lemma~\ref{intro:lem:K-1},
	there is a positive real number $C_2 \in \R_{>0}$ 
	such that for sufficiently large $m \in \Z_{>0}$,
	\[
	\glen_R\, H^1(X,\MN \otimes_{\MO_X} M_m) \le C_2 m^{d-1}
	\]
	holds.
	By the exact sequence 
	\[
	H^0(X,M_m)
	\to H^0(X_\red,M_m|_{X_\red})
	\to H^1(X,\MN \otimes_{\MO_X} M_m),
	\]
	we can show that
	the natural morphism 
	$H^0(X,M_m) \to H^0(X_\red,M_m|_{X_\red})$ 
	is not zero
	for sufficiently large $m$.
\end{proof}

To prove Theorem~\ref{intro:App-Kodaira's lemma},
the key lemma is the following.

\begin{lem}[Lemma~\ref{lem:App-main}]\label{intro:lem:App-main}
	Let $R$ be a noetherian ring, and
	let $M$ be a finitely generated $R$-module.
	For each $i \in \{1,2,\dots,n\}$,
	let $\Mp_i \in \Spec\,R$,
	let $R_i \coloneqq R / \Mp_i$, and
	let $M_i$ be an $R$-submodule of $M$.
	For each $i \in \{1,2,\dots,n\}$,
	assume the following conditions:  
	\begin{itemize}
		\item
			$\Mp_i (M/M_i) = 0$,
			in particular,
			$M/M_i$ is an $R_i$-module;
		\item
			if $\dim R_i = 0$, $\glen_{R_i} (M/M_i) > \log_2 n$;
		\item
			if $\dim R_i \ge 1$, $\glen_{R_i} (M/M_i) \ge 1$.
	\end{itemize}
	Then, $\bigcup_{i=1}^n M_i \neq M$ holds.
\end{lem}

\begin{proof}[Sketch of the proof of Lemma~\ref{intro:lem:App-main}]
	We can reduce the proof to the case where $M$ is finitely generated $\Z$-algebra.
	In this case,
	for each $i \in \{1,2,\dots,n\}$,
	pick a maximal ideal $\Mm_i$ of $R$ 
	such that $\Mp_i \subseteq \Mm_i$.
	Let $e$ be a sufficiently large integer.
	Let $\Mq \coloneqq \prod_{i=1}^n \Mm_i ^e$,
	let $M' \coloneqq M/\Mq M$, and
	let $M'_i \coloneqq f(M_i)$, 
	where $f:M \to M'$ be the natural surjection.
	By Lemma~\ref{lem:App-4}~(\ref{enu:App-4-2}),
	it is enough to show that
	$\bigcup_{i=1}^n M'_i \neq M'$.
	Then we can show the following conditions:
	\begin{itemize}
		\item
			$\#M' < \infty$;
		\item
			$\#(M' / M'_i) > n$.
	\end{itemize}
	Thus we have 
	\[
	\#\left(\bigcup_{i=1}^n M'_i \right) 
	\le \sum_{i=1}^n \#M'_i 
	= \sum_{i=1}^n \frac{\#M'}{\#(M'/M'_i)}
	< n \cdot \frac{\#M'}{n} 
	= \#M' 
	< \infty,
	\]
	which implies
	$\bigcup_{i=1}^n M'_i \neq M'$.
	
\end{proof}

\begin{ack}
The auther would like to thank Professor Hiromu Tanaka
for many comments and discussions.
\end{ack}


\section{Preliminaries}

\subsection{Notation}\label{sec:Notation}

\begin{enumerate}
	\item
		We freely use the notation and terminology in \cite{Har}.
		In particular,
		for the definitions of {\em projective morphisms} 
		and {\em quasi-projective morphisms} of schemes,
		we refer to \cite{Har}.
	\item
		Let $S$ be a set.
		$\#S$ denotes the number of the elements of $S$
		if $S$ is a finite set.
		If $S$ is an infinite set (resp. a finite set),
		we write $\#S = \infty$ (resp. $\#S < \infty$).
	\item
		Given a ring $A$,
		$f:M \to N$ is an {\em $A$-homomorphism}
		if it is a homomorphism of $A$-modules.
	\item
		Let $A$ be an integral domain.
		Let $\alpha:\Z \to A$ be the natural ring homomorphism.
		For a prime number $p$,
		$A$ is {\em of characteristic $p$}
		if the kernel of $\alpha$
		is the ideal $(p)$.
		$A$ is {\em of positive characteristic}
		if $A$ is of characteristic $p$ for some prime number $p$.
		$A$ is {\em of characteristic $0$}
		if $\alpha$ is injective.
	\item
		Given a field extension $k \to k'$,
		$\transdeg_k\; k'$ denotes transcendence degree of $k'$ over $k$.
	\item
		Let $X$ be an integral scheme.
		Then $K(X)$ denotes $\MO_{X,\xi}$,
		where $\xi$ is the generic point of $X$.
		Let $A$ be an integral domain. 
		Then $K(A)$ denotes $K(\Spec\,A)$.
	\item
		Given a closed subscheme $Y$ of a scheme $X$,
		$\Supp(Y)$ denotes 
		$Y$ cosidered as a closed subset of $X$.
	\item
		Given a morphism $f:X \to Y$ of schemes
		and a closed subscheme $X'$ of $X$,
		$f|_{X'}$ denotes the induced morphism 
		$X' \to X \stackrel{f}{\to} Y$.
	\item \label{enu:scheme-theoretic-union}
		Let $X$ be a noetherian scheme.
		Let $\sequence{Y}{,}{n}$ 
		(resp. $Y_\lambda$ ($\lambda \in \Lambda$))
		be closed subschemes of $X$.
		Then 
		the {\em scheme-theoretic union} 
		(resp. {\em scheme-theoretic intersection}) of 
		$\sequence{Y}{,}{n}$ 
		(resp. $Y_\lambda$ ($\lambda \in \Lambda$))
		is the closed subscheme of $X$ determined by the coherent ideal sheaf
		$\bigcap_{i=1}^n \MI_i$ (resp. $\sum_{\lambda \in \Lambda} \MI_\lambda$),
		where $\MI_i$ (resp. $\MI_\lambda$) 
		is the ideal sheaf of $Y_i$ (resp. $Y_\lambda$).
		%
	\item
		The notation $\cap$ (resp. $\cup$, $\subseteq$)
		denotes 
		the set-theoretic intersection (resp. union, inclusion) 
		except for Section~\ref{sec:K}.
		Given closed subschemes $Y_1$ and $Y_2$ of a scheme $X$,
		we write $Y_1 = Y_2$ 
		if $Y_1$ is equal to $Y_2$
		as closed subschemes.
		If confusion may arise,
		we use the notation $\Supp(-)$
		clearly to show that the scheme structure is ignored.
		%
	\item 
		Given a proper morphism $f:X \to Y$ of schemes
		and a closed subscheme $X'$ of $X$,
		$f(X')$ denotes the scheme-theoretic image of $f|_{X'}$
		(cf. \cite[Tag 01R7]{SP}).
		If $X'$ is integral (resp. irreducible, reduced),
		then also $f(X')$ is integral (resp. irreducible, reduced).
	\item
		Let $f:X \to Y$ be a morphism of schemes
		and let $y \in Y$.
		Then $X_y$ is defined as $X_y \coloneqq X \times_Y \Spec\, k(y)$,
		where $k(y)$ is the residue field of $y$ on $Y$.
		Let $\MF$ be an $\MO_X$-module on $X$.
		Then $\MF|_{X_y}$ denotes $g^* \MF$, 
		where $g:X_y \to X$ is the induced morphism.
	\item
		Given a scheme $X$,
		a quasi-coherent sheaf $\MF$ on $X$,
		and a section $s \in H^0(X, \MF)$,
		$s(x)$ denotes the image of $s$ to $\MF_x/\Mm_x \MF_x$,
		where $\Mm_x$ is the maximal ideal of $\MO_{X,x}$. 
		%
	\item
		Let $X$ be a scheme.
		$X_\red$ denotes the reduced scheme associated to $X$ 
		(cf. \cite[Exercises~II.2.3]{Har}).
		The {\em nilradical ideal sheaf} of $X$ 
		is the ideal sheaf of $X_\red$ on $X$.
	\item
		Let $X$ be a reduced scheme  
		and let $Y$ be an irreducible component of $X$.
		Then we view
		$Y$ to be equipped with the reduced induced structure.
	\item
		Given a scheme $X$, 
		quasi-coherent sheaves $\MF$ and $\MG$,
		an $\MO_X$-submodule $\MF'$ of $\MF$, and 
		a morphism $\phi: \MF \to \MG$ of $\MO_X$-module,
		$\phi(\MF')$ denotes the image sheaf of 
		the restricted homomorphism $\phi|_{\MF'}: \MF' \to \MG$.
	\item
		Let $f:X \to S$ be a proper morphism of noetherian schemes,
		and let $L$ be an invertible sheaf on $X$.
		$L$ is said to be {\em $f$-free} 
		if the natural morphism $f^* f_* L \to L$ is surjective.
		$L$ is said to be {\em $f$-very ample}
		if $L$ is $f$-free and
		the induced morphism $X \to \mathbb{P}(f_* L)$ is 
		a closed immersion.
		$L$ is said to be {\em $f$-semi-ample} (resp. {\em $f$-ample})
		if there exists a positive integer $m \in \Z_{>0}$ such that
		$L^{\otimes m}$ is $f$-free (resp. $f$-very ample).
		$L$ is said to be {\em $f$-nef} 
		if $L|_{X_s}$ is nef for any closed point $s \in S$.
		Then $L$ is $f$-nef 
		if and only if
		$L|_{X_s}$ is nef for any point $s \in S$ (cf. \cite[Lemma~2.6]{CT}).
		If $L$ is $f$-free (resp. $f$-very ample, etc.),
		we may simply say 
		$L$ is relatively free (resp. relatively very ample, etc.)
		or 
		$L$ is free (resp. very ample, etc.) over $S$.
	\item \label{enu:dim}
		Let $X$ be a scheme.
		$\dim X$ denotes the dimension of $X$.
		For the definition of the dimension of $X$,
		we refer to \cite{Har}
		if $X$ is not empty.
		If $X$ is empty,
		we define $\dim X = -\infty$.
	\item
		The phrase 
		``{\em \ldots\  for sufficiently large $m \in \Z_{>0}$}''
		means
		``there is an integer $m_0 \in \Z_{>0}$ 
		such that
		\ldots\ 
		for any $m \in \Z_{>0}$ with $m \ge m_0$''.
		The phrase 
		``{\em \ldots\  for sufficiently divisible $m \in \Z_{>0}$}''
		means
		``there is an integer $m_0 \in \Z_{>0}$ 
		such that
		\ldots\ 
		for any $m \in \{ m_0 k \mid k \in \Z_{>0}\}$''.
	
\end{enumerate}

\subsection{Relative dimension}\label{sec:RD}


\begin{dfn}
	Let $f:X \to S$ be a morphism of noetherian schemes,
	with $S$ irreducible.
	Let $\sigma$ be the generic point of $S$.
	Then, $\dim(X/S)$ denotes 
	the dimension of $X_\sigma$
	(cf. Subsection~\ref{sec:Notation}~(\ref{enu:dim})).
	In particular,
	if $X_\sigma$ is empty,
	then $\dim(X/S) = -\infty$.
\end{dfn}

\begin{rem}\label{RD-1}
	If $f$ is of finite type,
	then $X_\sigma$ is of finite type over the field $K(S_{\red})$,
	and hence $\dim(X/S) = \dim X_\sigma < \infty$ holds.
\end{rem}

\begin{prop}\label{RD-2}
	Let $f:X \to S$ be a morphism of noetherian schemes,
	with $S$ irreducible.
	Then, we have 
	$\dim(X/S) = \dim(X_{\red}/S_{\red})$.
\end{prop}

\begin{proof}
	Let $\sigma$ be the generic point of $S$.
	The assertion follows from the following commutative diagram:
	\[
	\begin{CD}
		X_\sigma 	@>>>	X_\red 	@>>>	X 			\\
		@VVV						@VVV					@VVV  \\
		S_\sigma 	@>>>	S_\red 	@>>>	S.
	\end{CD}
	\]
\end{proof}

\begin{prop}\label{RD-<}
	Let $f:X \to S$ be a morphism 
	of irreducible noetherian schemes.
	Let $Y$ be a closed subscheme of $X$
	such that $\Supp(Y) \subsetneq \Supp(X)$.
	Then, we have 
	$\dim(Y/S) \le \dim(X/S) - 1$.
\end{prop}

\begin{proof}
	By Proposition~\ref{RD-2},
	we may assume that $X$ and $S$ are reduced.
	Thus they are integral.
	If $f$ is not dominant,
	the assertion clearly holds.
	Thus, we may assume that $f$ is dominant.
	We may also assume that $X$ and $S$ are affine,
	by replacing them with their open affine subschemes.
	Note that also $Y$ is affine.
	%
	If $\dim(X/S) = \pm \infty$,
	the assertion is clear. 
	Thus we may assume that 
	$0 \le \dim(X/S) < \infty$.
	We now show the following claim.
	
	\begin{claim}\label{claim:RD-1}
	Let $\phi: A \to B$ be an injective ring homomorphism 
	of integral domains, and
	let $\psi: B \to C$ be a surjective ring homomorphism.
	If $\psi$ is not injective,
	then the induced ring homomorphism 
	\[
	\psi \otimes_{A} \id_{K(A)} : 
	B \otimes_{A} K(A) \to C \otimes_{A} K(A)
	\]
	is not injective.
	\end{claim}
	
	\begin{proof} [Proof of Claim~\ref{claim:RD-1}]
	Assume that $\psi$ is not injective.
	Let $I \coloneqq \Ker\, \psi \neq 0$.
	Then, we have the exact sequence
	\[
	0
	\to I (B \otimes_{A} K(A))
	\to B \otimes_{A} K(A)
	\to C \otimes_{A} K(A)
	\to 0.
	\]
	Since the induced ring homomorphism $\alpha:B \to B \otimes_{A} K(A)$
	is injective,
	we have 
	\[
	I (B \otimes_{A} K(A))
	\supseteq \alpha(I)
	\neq 0.
	\]
	Thus, the claim holds.
	\end{proof}
	
	Then, by applying Claim~\ref{claim:RD-1} to 
	$A  = \Gamma(S, \MO_S)$, 
	$B  = \Gamma(X, \MO_X)$,
	and $C = \Gamma(Y, \MO_Y)$,
	we have 
	$Y_\sigma \neq X_\sigma$.
	Since $X_\sigma$ is integral and 
	$Y_\sigma$ is a proper closed subscheme of $X_\sigma$,
	we have $\dim Y_\sigma \le \dim X_\sigma - 1$
	if $\dim X_\sigma \ge 1$.
	If $\dim X_\sigma = 0$, 
	then $Y_\sigma$ is empty for the same reason,
	and $\dim Y_\sigma \le \dim X_\sigma - 1$ holds.
\end{proof}

\subsection{Generic length}

\begin{dfn}\label{def:glen}
	\begin{enumerate}
		\item \label{enu:glen-alg}
			Let $A$ be a noetherian ring which has only one minimal prime ideal $\Mp$
			and let $M$ be a finitely generated $A$-module.
			Then, 
			we define the {\em generic $A$-length} $\glen_A\, M$ of $M$ as
			the length of $A_\Mp$-module $M_\Mp$.
		\item
			Let $X$ be an irreducible noetherian scheme
			and let $\MF$ be a coherent sheaf on $X$.
			Let $\xi \in X$ be the generic point of $X$.
			Then, 
			we define the {\em generic $X$-length} $\glen_X\, \MF$ of $\MF$ as
			the length of $\MO_{X, \xi}$-module $\MF_\xi$.
	\end{enumerate}
	
\end{dfn}

\begin{rem}\label{rem:GR-1}
	Let $A$, $\Mp$, and $M$ be as in Definition~\ref{def:glen}~(\ref{enu:glen-alg}).
	\begin{enumerate}
		\item
			Since $A_\Mp$ is artinian,
			$M_\Mp$ has finite length
			(cf. \cite[\S~3]{Mat}),
			i.e., $\glen_A\, M$ is a non-negative integer.
		\item
			Given an exact sequence of $A$-modules
			\[
			0 \to 
			\sequence{M}{\to}{n}
			\to 0,
			\]
			we have $\sum_{i=1}^n (-1)^i \glen_A\, M_i = 0$ 
			(cf. \cite[\S~2]{Mat}).
		\item
			If $\Mp M = 0$,
			$M$ is also an $A/\Mp$-module and
			we have 
			\[
			\glen_A\, M 
			= \glen_{A/\Mp}\, M.
			\]
		\item
			If $A$ is an integral domain,
			we have 
			$\glen_{A}\, M = \dim_{K(A)} (M \otimes_A K(A))$.
	\end{enumerate}
\end{rem}

\begin{rem} \label{rem:GR-2}
	Let $f:X \to S$ be a projective morphism 
	of noetherian schemes, with $S$ irreducible.
	Let $\MF$ be a coherent sheaf on $X$.
	Let $\sigma$ be the generic point of $S$,
	let $X_0 \coloneqq X \times_S \Spec\, \MO_{S, \sigma}$, and
	let $\MF_0 \coloneqq j^* \MF$, 
	where $j: X_0 \to X$ is the induced morphism. 
	Then, by \cite[Proposition~III.9.3]{Har},
	$\glen_S\, f_*\MF 
	= \glen_{\MO_{S,\sigma}}\, H^0(X_0,\MF_0)$ holds.
	If $S$ is affine, 
	we have 
	$\glen_S\, f_*\MF = \glen_A\, H^0(X,\MF)$,
	where $S = \Spec\, A$.
\end{rem}

\begin{rem}\label{GR-affineness}
	Let $f:X \to S$ be a projective morphism of noetherian schemes with $S$ irreducible, and
	let $\MF$ be a coherent sheaf on $X$.
	Let $S'$ be a non-empty open subscheme of $S$,
	let $X' = X \times_S S'$, and
	let $g:X' \to S'$ be the induced morphism.
	Then,
	we have
	$
	\glen_S\, f_* \MF = \glen_{S'}\, g_*(\MF|_{X'}).
	$
\end{rem}

\subsection{Associated points}


\begin{dfn}
	Let $A$ be a ring,
	let $M$ be an $A$-module, and
	let $\Mp \in \Spec\, A$.
	We say that 
	$\Mp$ is {\em associated to} $M$
	or $\Mp$ is an {\em associated prime} of $M$
	if there is an element $m \in M$ 
	whose annihilator $\{a \in A \mid am = 0 \}$ is equal to $\Mp$.
	$\Ass\, M$ denotes the set of all the associated primes of $M$.
\end{dfn}

\begin{dfn}
	Let $X$ be a scheme,
	let $\MF$ be a quasi-coherent sheaf on $X$,
	and let $x \in X$.
	\begin{enumerate}
		\item
			We say that 
			$x$ is {\em associated to} $\MF$
			or $x$ is an {\em associated point} of $\MF$
			if the maximal ideal of $\MO_{X,x}$ is associated to $\MO_{X,x}$-module $\MF_x$.
			We simply say $x$ is {\em associated to} $X$ 
			or $x$ is an {\em associated point} of $X$
			if $x$ is associated to $\MO_X$.
		\item
			$\Ass\, \MF$ denotes the set of all the associated points of $\MF$
			and $\Ass\, X$ denotes $\Ass\, \MO_X$.
	\end{enumerate}
\end{dfn}

\begin{prop}\label{prop:AP-1}
	Let $X$ be a locally noetherian scheme 
	and let $\MF$ be a quasi-coherent sheaf on $X$.
	Let $U = \Spec\, A \subseteq X$ be an open affine subscheme
	and let $M= H^0(U,\MF|_U)$.
	Let $x \in U$, 
	and let $\Mp \in \Spec\, A$ be the corresponding prime ideal.
	Then, $\Mp$ is associated to $M$ 
	if and only if $x$ is associated to $\MF$.
\end{prop}

\begin{proof}
	See \cite[Tag~02OK]{SP}.
\end{proof}

\begin{prop}
	Let $X$ be a noetherian scheme, and
	let $\MF$ be a coherent sheaf on $X$.
	Then, $\Ass\,\MF$ is a finite set.
\end{prop}

\begin{proof}
	See \cite[Tag~05AF]{SP}.
\end{proof}

\begin{prop}\label{prop:AP-2}
	Let $X$ be a locally noetherian scheme 
	and let $\MF$ be a quasi-coherent sheaf on $X$.
	\begin{enumerate}
		\item
			$\Ass\, \MF \subseteq \Supp\, \MF$ holds.
		\item
			Let $x \in \Supp\, \MF$ be a point 
			which is not a specialization of another point of $\Supp\, \MF$.
			Then, $x$ is associated to $\MF$.
			In particular, 
			all the generic points of the irreducible components of $X_\red$ are associated to $X$.
	\end{enumerate}
\end{prop}

\begin{proof}
	See \cite[Tag~05AD]{SP}
	and \cite[Tag~05AH]{SP}.
\end{proof}

\subsection{Zero loci}


The results in this subsection might be known to experts.
However, the auther could not have found appropriate references.
We include this subsection for the sake of completeness.

\begin{dfn}\label{def:ZL-1}
	Let $X$ be a noetherian scheme, 
	let $L$ be an invertible sheaf on $X$, and
	let $s \in H^0(X,L)$.
	Let $\phi: \MO_X \to L$ be the morphism of $\MO_X$-modules 
	such that 
	$\phi(U) (1) = s|_U$ holds
	for any open subset $U \subseteq X$.
	Let $\phi':L^{-1} \to \MO_X$ be the composition 
	\[
	L^{-1} 
	\stackrel{\simeq}{\to} L^{-1} \otimes_{\MO_X} \MO_X
	\stackrel{\id_{L^{-1}} \otimes \phi}{\longrightarrow} L^{-1} \otimes_{\MO_X} L
	\stackrel{\simeq}{\to} \MO_X.
	\]
	Let $\MI$ be the image sheaf of $\phi'$.
	Then, the {\em zero locus} $Z(s)$ of $s$ is defined as 
	the closed subscheme of $X$ 
	determined by the coherent ideal sheaf $\MI$.
\end{dfn}

\begin{prop}\label{prop:ZL-2}
	Let $X$ be a noetherian scheme, 
	let $L$ be an invertible sheaf on $X$, and
	let $s \in H^0(X,L)$.
	Then, 
	\[
	\Supp(Z(s)) = \{x \in X \mid s(x) = 0 \}
	\] 
	holds.
	In particular, we have 
	$\Supp(Z(s)) = \Supp(Z(s|_{X_{\red}}))$.
\end{prop}

\begin{proof}
	Let $\phi,\phi',\MI$ be as in Definition~\ref{def:ZL-1}.
	Let $x \in X$.
	The condition $x \in Z(s)$ is equivalent to the condition $\MI_x \subseteq \Mm_x$, 
	where $\Mm_x$ is the maximal ideal of $\MO_{X,x}$.
	Since the image sheaf of $\phi$ is $\MI L$, 
	$s_x$ generates $\MI_x L_x$ as an $\MO_{X,x}$-module.
	Thus, the condition $\MI_x \subseteq \Mm_x$ 
	is equivalent to 
	the condition $s_x \in \Mm_x L_x$,
	i.e., $s(x) = 0$.
\end{proof}

\begin{prop}\label{prop:ZL-3}
	Let $X$ be a reduced noetherian scheme,
	let $L$ be an invertible sheaf on $X$, and
	let $s \in H^0(X,L)$.
	Then, the following conditions are equivalent:
	\begin{enumerate}
		\item
			$s=0$; 
		\item
			$s(x) = 0$ for all $x \in X$;
		\item
			$\Supp(Z(s)) = \Supp(X)$;
		\item
			$Z(s) = X$.
	\end{enumerate}
\end{prop}

\begin{proof}
	It is clear that (1) implies (2).
	By Proposition~\ref{prop:ZL-2},
	(2) is equivalent to (3).
	Since $X$ is reduced,
	(3) implies (4).
	Thus it is enough to show that
	(4) implies (1).
	Let $\phi, \MI$ be as in Definition~\ref{def:ZL-1}.
	Since $Z(s) = X$, 
	we have $\MI = 0$, 
	which implies $\phi$ is zero,
	i.e., $s=0$.
\end{proof}

\begin{dfnprop}\label{prop:ZL-effective divisor}
	Let $X$ be a noetherian scheme, 
	let $L$ be an invertible sheaf on $X$, and
	let $s \in H^0(X,L)$.
	Let $\phi,\phi',\MI$ be as in Definition~\ref{def:ZL-1}.
	Then, the following conditions are equivalent:
	\begin{enumerate}
		\item
			$\MI \simeq L^{-1}$ holds as $\MO_X$-modules;
		\item
			$\MI$ is an invertible sheaf on	$X$;
		\item
			$\phi'$ is injective;
		\item
			$\phi$ is injective;
		\item
			for any $x \in \Ass\, X$, 
			$s(x) \neq 0$ holds;
		\item
			for any $x \in \Ass\, X$, 
			$s(y) \neq 0$ holds
			for some $y \in \overline{\{x\}}$,
			where $\overline{\{x\}}$ is the closure of the subset $\{x\}$ of $X$;
		\item
			for any $x \in \Ass\, X$, 
			$\overline{\{x\}} \nsubseteq \Supp(Z(s))$ holds,
			where $\overline{\{x\}}$ is the closure of the subset $\{x\}$ of $X$.
	\end{enumerate}
	If $s$ satisfies the above equivalent conditions,
	we say that 
	$Z(s)$ {\em defines an effective Cartier divisor}.
\end{dfnprop}

\begin{proof}
	It is clear that 
	(3) implies (1), (1) implies (2), and (3) is equivalent to (4).
	
	First, we show that (2) implies (3).
	Since the conditions are local,
	we may assume that 
	$X$ is affine and 
	written as $X = \Spec\,A$ with a noetherian ring $A$.
	We may also assume that
	$\MI$ and $L$ are isomorphic to $\MO_X$.
	Let $\MJ$ be the kernel sheaf of the composition 
	\[
	\psi:\MO_X \stackrel{\simeq}{\to} L^{-1} \stackrel{\phi'}{\to} \MO_X,
	\]
	and let $J \coloneqq \Gamma(X,\MJ)$.
	Then, $\psi(X)$ induces the isomorphism 
	\[
	A/J \simeq \Gamma(X,\MI) \simeq A
	\]
	as $A$-modules.
	Multiplying the above $A$-modules by the ideal $J$,
	we have 
	\[
	0 = J (A/J) \simeq JA = J.
	\]
	Thus, $\MJ$ is the zero sheaf and $\phi'$ is injective.
	
	Second, we show that (4) is equivalent to (5).
	Let $U \simeq \Spec\,A$ be an open affine subscheme of $X$
	such that $L|_U \simeq \MO_U$.
	Let $M \coloneqq \Gamma(U,L|_U)$.
	Then, $\phi|_U$ is injective
	if and only if
	$s|_U$ is a non-zero-divisor of the $A$-module $M$.
	Since $M \simeq A$,
	$s|_U$ is a non-zero-divisor of $M$
	if and only if $s|_U \notin \Mp M$ holds for every $\Mp \in \Ass\,A$.
	By Proposition~\ref{prop:AP-1},
	it is equivalent to 
	the condition that $s_x \notin \Mm_x L_x$ holds for every $x \in (\Ass\,X) \cap U$,
	where $\Mm_x$ is the maximal ideal of $\MO_{X,x}$.
	Thus, we have that 
	$\phi$ is injective 
	 if and only if $s_x \notin \Mm_x L_x$, 
	 i.e., $s(x) = 0$, 
	 holds for every $x \in \Ass\,X$.
	
	Finally, we show that (5), (6), and (7) are equivalent.
	It is clear that (5) implies (6).
	Let $x \in \Ass\,X$ and 
	let $y \in \overline{\{x\}}$.
	Then, $s(x) = 0$ implies $s(y)= 0$.
	Thus, we have that (6) implies (5).
	By Proposition~\ref{prop:ZL-2},
	(6) is equivalent to (7).
\end{proof}

\begin{prop}\label{prop:ZL-4}
	Let $X$ be a noetherian scheme, 
	let $L_1$ and $L_2$ be  invertible sheaves on $X$, 
	let $s_1 \in H^0(X,L_1)$, and 
	let $s_2 \in H^0(X,L_2)$.
	Let $\MI$ (resp. $\MI_1$, $\MI_2$) be the ideal sheaf of $Z(s_1 \otimes s_2)$ (resp. $Z(s_1)$, $Z(s_2)$).
	Then, 
	$\MI = \MI_1 \MI_2$ holds
	as sheaves of ideals on $X$.
\end{prop}

\begin{proof}
	Let $\phi':L_1^{-1} \otimes_{\MO_X} L_2^{-1} \to \MO_X$ 
	(resp. ${\phi_1':L_1^{-1} \to \MO_X}$, ${\phi_2':L_2^{-1} \to \MO_X}$) be 
	the morphism	of $\MO_X$-modules 
	corresponding to $s_1 \otimes s_2$ 
	(resp. $s_1$, $s_2$) as in Definition~\ref{def:ZL-1}.
	Then, $\phi'$ coincides with the composition
	\[
	L_1^{-1} \otimes_{\MO_X} L_2^{-1}
	\stackrel{\phi_1' \otimes \id_{L_2^{-1}}}{\longrightarrow}
	\MO_X \otimes_{\MO_X} L_2^{-1}
	\stackrel{\simeq}{\to}
	L_2^{-1}
	\stackrel{\phi_2'}{\to}
	\MO_X.
	\]
	Thus, we have
	$\MI 
	= \phi'(L_1^{-1} \otimes_{\MO_X} L_2^{-1})
	= \phi_2'(\MI_1 L_2^{-1}) 
	= \MI_1 \phi_2'(L_2^{-1})
	= \MI_1 \MI_2$.
\end{proof}

\begin{prop}\label{prop:ZL-5}
	Let $X$ be a noetherian scheme, 
	let $Y$ be a closed subscheme of $X$,
	let $L$ be an invertible sheaf on $X$, and 
	let $s \in H^0(X,L)$.
	Then,
	$Z(s|_Y)$
	coincides with 
	the scheme-theoretic intersection of  $Z(s)$ and $Y$.
\end{prop}

\begin{proof}
	Let $\phi':L^{-1} \to \MO_X$ 
	(resp. ${\phi_Y':(L|_Y)^{-1} \to \MO_Y}$)
	be the morphism of $\MO_X$-modules (resp. $\MO_Y$-modules)
	corresponding to $s$ 
	(resp. $s|_Y$) as in Definition~\ref{def:ZL-1}.
	Let $\MI$ (resp. $\MI'$)
	be the ideal sheaf of $Z(s)$ 
	(resp. $Z(s|_Y)$) on $X$.
	Let $\MI_Y$ be the ideal sheaf of $Y$ on $X$.
	We consider the following commutative diagram:
	\[
	\begin{CD}
		L^{-1}			@>\phi'>>		\MO_X \\
		@VV\beta V								@VV\alpha V \\
		(L|_Y)^{-1}	@>\phi_Y'>>	\MO_Y. 
	\end{CD}
	\]
	Note that $\alpha$ and $\beta$ are surjective.
	Since $\alpha (\MI')$ is the ideal sheaf of $Z(s|_Y)$ on $Y$,
	we have 
	\[
	\alpha (\MI')
	= \phi_Y'((L|_Y)^{-1}) 
	= (\phi_Y' \circ \beta) (L ^{-1}) 
	= (\alpha \circ \phi') (L ^{-1})
	= \alpha (\MI).
	\]
	Since $\MI' \supseteq \MI_Y$,
	we have 
	\[
	\MI' = \MI' + \MI_Y = \MI + \MI_Y.
	\]
\end{proof}

\subsection{Base loci}

\begin{dfn} \label{def:base-locus}
	Let $f:X \to S$ be a projective morphism 
	of noetherian schemes.
	Let $L$ be an invertible sheaf on $X$.
	\begin{itemize}
		\item
			The {\em $f$-base locus} $\Bs_f(L)$ of $L$ 
			is defined as 
			\[
			\Bs_f(L) \coloneqq \Supp \  \Coker(f^* f_* L \to L).
			\]
		\item
			The {\em $f$-stable base locus} $\SB_f(L)$ of $L$ 
			is defined as 
			\[
			\SB_f(L) \coloneqq \bigcap_{m \in \Z_{>0}} \Bs_f(L^{\otimes m}).
			\]
		\item
			The {\em $f$-augmented base locus} $\B_{+,f}(L)$ of $L$ 
			is defined as 
			\[
			\B_{+,f}(L) 
			\coloneqq \bigcap_{m \in \Z_{>0}} \SB_f(L^{\otimes m} \otimes_{\MO_X} A^{-1}),
			\]
			where $A$ is an $f$-ample invertible sheaf on $X$.
	\end{itemize}
	$\Bs_f(L)$, $\SB_f(L)$ and $\B_{+,f}(L)$ are closed subsets of $X$.
\end{dfn}

The following argument is extracted from \cite{CMM}.

\begin{rem}
	Since $X$ is a noetherian topological space,
	for sufficiently divisible $m \in \Z_{>0}$,
	\[
	\SB_f(L) = \Bs_f(L^{\otimes m})
	\]
	holds.
	For the same reason,
	for sufficiently large $m \in \Z_{>0}$,
	\[
	\B_{+,f}(L) = \SB_f(L^{\otimes m} \otimes_{\MO_X} A^{-1})
	\]
	holds.
\end{rem}

The following argument is extracted from \cite{CMM}.

\begin{rem}
	The definition of $\B_{+,f}(L)$ is independent of the choice of $A$.
	Indeed, if $A_1$ and $A_2$ are $f$-ample invertible sheaves,
	then there is a positive integer $r \in \Z_{>0}$ such that $A_1^{\otimes r} \otimes_{\MO_X} A_2^{-1}$ is $f$-free.
	Then, by
	\[
	L^{\otimes mr} \otimes_{\MO_X} A_2^{-1} 
	=(L^{\otimes m} \otimes_{\MO_X} A_1^{-1})^{\otimes r} \otimes_{\MO_X} (A_1^{\otimes r} \otimes_{\MO_X} A_2^{-1}),
	\]
	we have
	$\SB_f(L^{\otimes mr} \otimes_{\MO_X} A_2^{-1}) 
	\subseteq \SB_f(L^{\otimes m} \otimes_{\MO_X} A_1^{-1})$.
	By the same argument, we also have 
	$\SB_f(L^{\otimes mr'} \otimes_{\MO_X} A_1^{-1}) 
	\subseteq \SB_f(L^{\otimes m} \otimes_{\MO_X} A_2^{-1})$
	for some $r' \in \Z_{>0}$.
	Thus, the following holds:
	\[
	\bigcap_{m \in \Z_{>0}} \SB_f(L^{\otimes m} \otimes_{\MO_X} A_1^{-1})
	= \bigcap_{m \in \Z_{>0}} \SB_f(L^{\otimes m} \otimes_{\MO_X} A_2^{-1}).
	\]
\end{rem}

\begin{prop}\label{BL-affineness}
	Let $f:X \to S$ be a projective morphism of noetherian schemes, and
	let $L$ be an invertible sheaf on $X$.
	Let $S'$ be an open subscheme of $S$,
	let $X' \coloneqq X \times_S S'$, and
	let $g:X' \to S'$ be the induced morphism.
	\begin{enumerate}
		\item\label{enu:BL-affineness-1}
			$\Bs_f(L) \cap X' = \Bs_g(L|_{X'})$ holds.
		\item\label{enu:BL-affineness-2}
			$\SB_f(L) \cap X' = \SB_g(L|_{X'})$ holds.
		\item\label{enu:BL-affineness-3}
			$\B_{+,f}(L) \cap X' = \B_{+,g}(L|_{X'})$ holds.
	\end{enumerate}
\end{prop}

\begin{proof}
	By the definition of relative base loci,
	we have (\ref{enu:BL-affineness-1}).
	(\ref{enu:BL-affineness-2}) follows from (\ref{enu:BL-affineness-1}).
	(\ref{enu:BL-affineness-3}) follows from (\ref{enu:BL-affineness-2}).
\end{proof}

\begin{prop}\label{BL-Z(s)}
	Let $f:X \to S$ be a projective morphism of noetherian schemes,
	with $S$ affine.
	Let $L$ be an invertible sheaf on $X$.
	Then, we have 
	\[
	\Bs_f(L) = \bigcap_{s \in H^0(X,L)} \Supp(Z(s)).
	\]
\end{prop}

\begin{proof}
	Let $x \in X$.
	Then, $x \notin \Bs_f(L)$
	if and only if 
	$(f^* f_* L)_x \to L_x$ is surjective,
	i.e., $s(x) \neq 0$ for some $s \in H^0(X,L)$.
	By Proposition~\ref{prop:ZL-2},
	$s(x) \neq 0$
	if and only if
	$x \notin \Supp(Z(s))$.
	Thus, the assertion holds.
%
%
%
\end{proof}

\begin{prop}\label{prop:BL-|_Y}
	Let $f:X \to S$ be a projective morphism 
	of noetherian schemes 
	and let $L$ be an invertible sheaf on $X$.
	Let $Y$ be a closed subscheme of $X$.
	\begin{enumerate}
		\item \label{enu:BL-|_Y-Bs}
			$\Bs_{f|_Y}(L|_Y) \subseteq \Bs_f(L)$ holds.
		\item \label{enu:BL-|_Y-SB}
			$\SB_{f|_Y}(L|_Y) \subseteq \SB_f(L)$ holds.
		\item \label{enu:BL-|_Y-B_+}
			$\B_{+,f|_Y}(L|_Y) \subseteq \B_{+,f}(L)$ holds.
	\end{enumerate}
\end{prop}

The following proof is extracted from \cite{CMM}.

\begin{proof}
	By Proposition~\ref{BL-affineness},
	we may assume that $S$ is affine.
	Then we have
	\begin{align*}
	\Bs_{f|_Y}(L|_Y) 
	&= \bigcap_{s \in H^0(Y, L|_Y)} \Supp(Z(s)) \\
	&\subseteq \bigcap_{t \in H^0(X,L)} \Supp(Z(t|_Y)) \\
	&= \bigcap_{t \in H^0(X,L)} (\Supp(Z(t)) \cap \Supp(Y)),
	\end{align*}
	where the first equality follows from Proposition~\ref{BL-Z(s)}
	and the last equality follows from	 Proposition~\ref{prop:ZL-5}.
	On the other hand, we have
	\[
	\bigcap_{t \in H^0(X,L)} (\Supp(Z(t)) \cap \Supp(Y))  
	\subseteq \bigcap_{t \in H^0(X,L)} \Supp(Z(t))
	= \Bs_f(L),
	\]
	where the last equality follows from Proposition~\ref{BL-Z(s)}.
	Hence (\ref{enu:BL-|_Y-Bs}) holds.
	(\ref{enu:BL-|_Y-SB}) follows from (\ref{enu:BL-|_Y-Bs}).
	(\ref{enu:BL-|_Y-B_+}) follows from (\ref{enu:BL-|_Y-SB}).
\end{proof}

\subsection{Irredundant decomposition of noetherian schemes}


\begin{prop}\label{Dec-sheaf}
	Let $X$ be a noetherian scheme, and
	let $\MF$ be a coherent sheaf on $X$.
	Let $\sequence{x}{,}{n}$ be all the associated points of $\MF$.
	Let $I \coloneqq \{1,2,\dots,n\}$.
	Then, 
	there are coherent subsheaves $\sequence{\MG}{,}{n}$ of $\MF$
	such that the following conditions hold:
	\begin{enumerate}
		\item
			$\bigcap_{i \in I} \MG_i = 0$;
		\item
			$\Ass(\MF / \MG_i) = \{x_i\}$ for every $i \in I$;
		\item
			$\bigcap_{j \in J} \MG_j \neq 0$ 
			for any proper subset $J \subsetneq I$.
	\end{enumerate}
\end{prop}

\begin{proof}
	See \cite[Proposition~3.2.6]{EGA-IV-2}.
\end{proof}

\begin{dfnprop}\label{Dec-scheme}
	Let $X$ be a noetherian scheme.
	Let $\sequence{x}{,}{n}$ be all the associated points of $X$.
	Let $I \coloneqq \{1,2,\dots,n\}$.
	Then, 
	there are closed subschemes $\sequence{X}{,}{n}$ of $X$
	such that the following conditions hold:
	\begin{enumerate}
		\item
			$X$ coincides with 
			the scheme-theoretic union of $\sequence{X}{,}{n}$
			(cf. Subsection~\ref{sec:Notation}~(\ref{enu:scheme-theoretic-union}));
		\item
			$\Ass\, X_i = \{x_i\}$ for every $i \in I$;
		\item
			for any proper subset $J \subsetneq I$,
			$X$ does not coincide with 
			the scheme-theoretic union $\bigcup_{j \in J} X_j$.
	\end{enumerate}
	If $\sequence{X}{,}{n}$ satisfy the above conditions,
	the decomposition $X = \bigcup_{i \in I} X_i$,
	where the union	is the scheme-theoretic union, 
	is called
	an {\em irredundant decomposition} of $X$.
\end{dfnprop}

\begin{proof}
	Apply Proposition~\ref{Dec-sheaf} to $\MO_X$.
\end{proof}

\section{Weak bigness and exceptional loci}

\subsection{Weak bigness}


\begin{dfn} 
Let $f:X \to S$ be a proper morphism of noetherian schemes
and let $L$ be an invertible sheaf on $X$.
Let $A$ be an $f$-ample invertible sheaf on $X$.
$L$ is {\em $f$-weakly big} 
if there exists a positive integer $m \in \Z_{>0}$ 
such that 
\[
 (f|_{X_\red})_*((L^{\otimes m} \otimes_{\MO_X} A^{-1})|_{X_{\red}}) \neq 0.
\]
When $X$ is irreducible, we simply say $L$ is {\em $f$-big} 
if $L$ is $f$-weakly big.
\end{dfn}

\begin{rem}
	The definition of $f$-weak bigness is independent of the choice of $A$.
	To prove this,
	we show that
	for any $f$-ample invertible sheaves $A_1$ and $A_2$,
	if
	\[
	 (f|_{X_\red})_*((L^{\otimes m} \otimes_{\MO_X} A_1^{-1})|_{X_{\red}}) \neq 0
	\]
	holds for some $m \in \Z_{>0}$,
	then
	\[
	 (f|_{X_\red})_*((L^{\otimes m'} \otimes_{\MO_X} A_2^{-1})|_{X_{\red}}) \neq 0
	\]
	holds for some $m' \in \Z_{>0}$.
	By replacing $X$ with $X_\red$,
	we may assume that $X$ is reduced.
	Assume that
	$f_*(L^{\otimes m} \otimes_{\MO_X} A_1^{-1}) \neq 0$.
	Then, there are an open affine subset $U$ of $S$
	and a nonzero element 
	\[
	s 
	\in H^0(U, f_*(L^{\otimes m} \otimes_{\MO_X} A_1^{-1}))
	= H^0(f^{-1}(U), L^{\otimes m} \otimes_{\MO_X} A_1^{-1}).
	\]
	Pick a closed point $x \in f^{-1}(U)$
	such that
	$s(x) \neq 0$.
	Since $A_1$ is $f$-ample,
	there is a positive integer $r \in \Z_{>0}$
	such that $A_1^{\otimes r} \otimes_{\MO_X} A_2^{-1}$ is $f$-free.
	Thus, there is an element 
	\[
	t 
	\in H^0(U, f_*(A_1^{\otimes m} \otimes_{\MO_X} A_2^{-1}))
	= H^0(f^{-1}(U), A_1^{\otimes m} \otimes_{\MO_X} A_2^{-1})
	\]
	such that
	$t(x) \neq 0$.
	Then, by
	\[
	L^{\otimes mr} \otimes_{\MO_X} A_2^{-1} 
	=(L^{\otimes m} \otimes A_1^{-1})^{\otimes r} 
	\otimes_{\MO_X} (A_1^{\otimes m} \otimes A_2^{-1}),
	\]
	$s^{\otimes r} \otimes t$
	is an element of 
	\[
	H^0(f^{-1} (U), L^{\otimes mr} \otimes_{\MO_X} A_2^{-1})
	= H^0(U, f_* (L^{\otimes mr} \otimes_{\MO_X} A_2^{-1})).
	\]
	Since $(s^{\otimes r} \otimes t) (x) \neq 0$,
	we have 
	$f_* (L^{\otimes mr} \otimes_{\MO_X} A_2^{-1}) \neq 0$.
\end{rem}

\begin{prop}\label{prop:WB-affineness}
	Let $f:X \to S$ be a projective morphism of noetherian schemes, and
	let $L$ be an invertible sheaf on $X$.
	Then, the following conditions are equivalent:
	\begin{enumerate}
		\item \label{enu:WB-affineness-1}
			$L$ is $f$-weakly big;
		\item \label{enu:WB-affineness-2}
			For some non-empty open subscheme $S'$ of $S$,
			$L|_{X'}$ is $g$-weakly big,
			where $X' \coloneqq X \times_S S'$ 
			and $g:X' \to S'$ is the natural morphism.
		\item \label{enu:WB-affineness-3}
			For some non-empty open affine subscheme $S'$ of $S$,
			$L|_{X'}$ is $g$-weakly big,
			where $X' \coloneqq X \times_S S'$ 
			and $g:X' \to S'$ is the natural morphism.
	\end{enumerate}
\end{prop}

\begin{proof}
	%
	By the definition of weak bigness, 
	we have that 
	(\ref{enu:WB-affineness-1}) is equivalent to (\ref{enu:WB-affineness-2}).
	Since $S$ has an open affine cover,
	it also holds that (\ref{enu:WB-affineness-1}) is equivalent to (\ref{enu:WB-affineness-3})
	by the definition of sheaves.
\end{proof}


\begin{prop}\label{prop:WB-image}
	%
	Let $g:X \to S'$ be a projective morphism of noetherian schemes,
	let $\alpha:S' \to S$ be a finite morphism of noetherian schemes, and
	let $f \coloneqq \alpha \circ g$.
	Let $L$ be an invertible sheaf on $X$.
	Then,
	$L$ is $f$-weakly big
	if and only if
	$L$ is $g$-weakly big.
\end{prop}

\begin{proof}
	By Proposition~\ref{prop:WB-affineness},
	we may assume that $S'$ and $S$ are affine.
	Let $A$ be an $f$-ample invertible sheaf on $X$.
	Note that $A$ is also $g$-ample.
	Then, $L$ is $f$-weakly big  if and only if
	$H^0(X_\red,(L^{\otimes m} \otimes_{\MO_X} A^{-1})|_{X_{\red}}) \neq 0$
	for some $m \in \Z_{>0}$,
	which is equivalent to $g$-weak bigness of $L$.
\end{proof}

%

\begin{prop} \label{WB-lower bound} 
Let $f:X \to S$ be a projective surjective morphism of integral noetherian schemes, and
let $L$ be an invertible sheaf on $X$.
Let $\sigma$ be the generic point of $S$, and
let $d \coloneqq \dim(X/S)$.
Then, the following conditions are equivalent:
\begin{enumerate}
\item
	$L$ is $f$-big;
\item
	$L|_{X_\sigma}$ is $g$-big,
	where $g:X_\sigma \to \Spec\, K(S)$ is the induced morphism;
\item
	there is a positive real number $C \in \R_{>0}$
	such that for sufficiently large $m \in \Z_{>0}$,
	$\glen_S\, f_*L^{\otimes m} \ge Cm^d$ holds. 
\end{enumerate}
\end{prop}

\begin{proof}
	Let $A$ be an $f$-ample invertible sheaf on $X$.
	Note that $A|_{X_\sigma}$ is ample.
	By Remark~\ref{GR-affineness} and Proposition~\ref{prop:WB-affineness},
	we may assume that $S$ is affine and 
	written as $S = \Spec\,R$ with a noetherian ring $R$.
	By \cite[Proposition~III.9.3]{Har},
	we have 
	\[
	H^0(X,L^{\otimes m} \otimes_{\MO_X} A^{-1}) \otimes_{R} K(R) 
	= H^0(X_\sigma, (L|_{X_\sigma})^{\otimes m} \otimes_{\MO_{X_\sigma}} (A|_{X_\sigma})^{-1})
	\] 
	for any $m \in \Z_{>0}$.
	Thus, (1) is equivalent to (2).
	By Remark~\ref{rem:GR-1} and Remark~\ref{rem:GR-2},
	\[
	\glen_S\, f_*L^{\otimes m} 
	= \dim_{K(S)} H^0(X_\sigma, (L|_{X_\sigma})^{\otimes m})
	\]
	holds.
	Thus by \cite[Lemma~4.3]{Bir17},
	(2) is equivalent to (3).
\end{proof}

\begin{prop}\label{WB-affineness2}
	Let $f:X \to S$ be a projective morphism of noetherian schemes, 
	with $X$ irreducible.
	Let $L$ be an invertible sheaf on $X$.
	Then, the following conditions are equivalent:
	\begin{enumerate}
		\item \label{enu:WB-affineness2-1}
			$L$ is $f$-big;
		\item \label{enu:WB-affineness2-2}
			For every open subscheme $S'$ of $S$
			with $S' \cap f(X) \neq \emptyset$,
			$L|_{X'}$ is $g$-big,
			where $X' \coloneqq X \times_S S'$ 
			and $g:X' \to S'$ is the natural morphism.
	\end{enumerate}
\end{prop}

\begin{proof}
	Let $X'$, $S'$, and $g$ be as in (\ref{enu:WB-affineness2-2}).
	Note that $X'$ is not empty and it is irreducible.
	Replacing $X$ with $X_\red$,
	we may assume that $X$ is integral.
	Then also $X'$ is integral.
	By Proposition~\ref{prop:WB-image},
	replacing $S$ with $f(X)$,
	we may assume that 
	$f$ is surjective and $S$ is integral.
	Then also $S'$ is integral.
	Since
	$\glen_S\, f_*L^{\otimes m} 
	= \glen_{S'}\, g_*L^{\otimes m}$
	by Remark~\ref{GR-affineness},
	the assertion follows from Proposition~\ref{WB-lower bound},
\end{proof}

\begin{lem}\label{WB-upper bound}
	Let $f:X \to S$ be a projective morphism of noetherian schemes,
	with $S$ integral.
	Let $L$ be an invertible sheaf on $X$, and
	let $\MF$ be a coherent sheaf on $X$.
	Let $d \coloneqq \dim(X/S)$.
	Then, there exists a positive real number $C \in \R_{>0}$ such that 
	for sufficiently large $m \in \Z_{>0}$, 
	we have
	\[
	\glen_S\, f_* (\MF \otimes_{\MO_X} L^{\otimes m}) \leq Cm^d .
	\]
\end{lem}

\begin{proof}
	Let $\sigma$ be the generic point of $S$.
	By Remark~\ref{rem:GR-2},
	we have
	\[
	\glen_S\, f_* (\MF \otimes_{\MO_X} L^{\otimes m})
	= \dim_{K(S)} H^0(X_\sigma, (\MF|_{X_\sigma}) \otimes_{ \MO_{X_\sigma}} (L|_{X_\sigma})^{\otimes m}).
	\]
	If $d \ge 0$,
	by \cite[Lemma~4.1]{Bir17},
	there exists a positive real number $C \in \R_{>0}$ such that 
	\[
	\dim_{K(S)} H^0(X_\sigma, (\MF|_{X_\sigma}) \otimes_{ \MO_{X_\sigma}} (L|_{X_\sigma})^{\otimes m})
	\le Cm^d
	\]
	holds for sufficiently large $m \in \Z_{>0}$.
	If $X_\sigma = \emptyset$,
	\[
	H^0(X_\sigma, (\MF|_{X_\sigma}) \otimes_{ \MO_{X_\sigma}} (L|_{X_\sigma})^{\otimes m}) = 0
	\] 
	holds.
	Thus, the above inequality also holds when $d=-\infty$.
\end{proof}

The following proof is the same as in \cite[Lemma~2.6]{Sti21}
and almost the same as in \cite[Lemma~2.3]{CMM}.

\begin{lem}\label{E-main}
Let $f:X \to S$ be a projective morphism
of noetherian schemes,
with $X$ reduced and $S$ affine.
Let $L$ be an invertible sheaf on $X$, 
and let $A$ be an $f$-ample invertible sheaf on $X$.
Let $X_1$ be an irreducible component of $X$
such that $L|_{X_1}$ is $f|_{X_1}$-big.
Then, 
for sufficiently large $m \in \Z_{>0}$,
there is a section $s \in H^0(X,L^{\otimes m} \otimes A^{-1})$
such that
$\Supp(X_1) \nsubseteq \Supp(Z(s))$ holds.
\end{lem}

\begin{proof}
Assume that
$\Supp(X_1) \subseteq \Supp(Z(s))$ holds
for infinitely many $m \in \Z_{>0}$ and
for every $s \in H^0(X,L^{\otimes m} \otimes_{\MO_X} A^{-1})$.
Let $X'$ be the union of the other irreducible components of $X$ equipped with the  reduced scheme structure.
Let $T$ be the scheme-theoretic intersection of $X_1$ and $X'$.
Let $S_1 \coloneqq f( X_1 )$.
Note that $S_1$ is an integral noetherian affine scheme.
Thus, we can write $S_1 = \Spec\,R_1 $ with a noetherian integral domain $R_1$.
Let $d \coloneqq \dim(X_1/ S_1)$.
Let $M \coloneqq L^{\otimes m} \otimes_{\MO_X} A^{-1}$.
By the Mayer--Vietoris sequence, we have the exact sequence
\[
 0 \to \MO_X \to \MO_{X_1} \oplus \MO_{X'} \to \MO_T \to 0. 
\]
After tensoring with $M$ and taking global section, this induces the exact sequence
\[
 0 \to H^0(X,M) \to H^0({X_1},M|_{X_1}) \oplus H^0({X'},M|_{X'}) \to H^0({T},M|_T). 
\]
By the assumption, 
the homomorphism $H^0(X,M) \to H^0({X_1},M|_{X_1})$ is zero for infinitely many $m \in \Z_{>0}$.
Thus, for such $m$, 
the homomorphism
$H^0({X_1},M|_{X_1}) \to H^0(T,M|_T)$ is injective, 
and we have
\begin{equation}\label{E-ieq1}
\glen_{R_1}\, H^0({X_1},M|_{X_1}) \leq \glen_{R_1}\, H^0(T,M|_T).
\end{equation}
By Proposition~\ref{RD-<},
we have $\dim(T/S_1) \le d-1$.
By Lemma~\ref{WB-upper bound}, 
there is a positive real number $C \in \R _{>0}$ 
such that 
for sufficiently large $m>0$,
\[
 \glen_{R_1}\, H^0(T,M|_T) 
 \leq Cm^{\dim(T/S_1)} 
 \leq Cm^{d-1}
\]
holds.
On the other hand, 
by Proposition~\ref{WB-lower bound},
there is a positive real number $C' \in \R _{>0}$
such that
for sufficiently large $m>0$, 
\[
 \glen_{R_1}\, H^0({X_1},M|_{X_1}) \geq C' m^{d}
\]
holds.
These two estimates contradict inequality~(\ref{E-ieq1}).
\end{proof}

\begin{thm}\label{WB-wb=b}
	Let $f:X \to S$ be a projective morphism of noetherian schemes,
	and let $L$ be an invertible sheaf on $X$.
	Then, the following conditions are equivalent:
	\begin{enumerate}
		\item
			$L$ is $f$-weakly big;
		\item
			$L|_{X_\red}$ is $f|_{X_\red}$-weakly big;
		\item
			there exists an irreducible component $Y$ of $X_\red$
			such that $L|_Y$ is $f|_Y$-big.
	\end{enumerate}	
\end{thm}

\begin{proof}
	By Proposition~\ref{prop:WB-affineness},
	we may assume that $S$ is affine.
	By the definition of $f$-weak bigness,
	it is clear that (1) is equivalent to (2).
	By Lemma~\ref{E-main}, (3) implies (2).
	We now show that (2) implies (3).
	Replacing $X$ with $X_\red$,
	we may assume that $X$ is reduced.
	Let $\sequence{X}{,}{n}$ be all the irreducible components of $X$.
	Let $A$ be an $f$-ample invertible sheaf on $X$.
	Assume that
	for any $i\in \{1,2,\dots,n\}$ and $m \in \Z_{>0}$,
	\[
	H^0(X_i, (L^{\otimes m} \otimes_{\MO_X} A^{-1})|_{X_i} ) = 0
	\]
	holds.
	By the Mayer--Vietoris sequence,
	we have the injective homomorphism 
	\[
	H^0(X,L^{\otimes m} \otimes_{\MO_X} A^{-1})
	\to 
	\bigoplus_{i=1}^n H^0(X_i, (L^{\otimes m} \otimes_{\MO_X} A^{-1})|_{X_i} ).
	\]
	Thus, we have 
	$H^0(X,L^{\otimes m} \otimes_{\MO_X} A^{-1}) = 0$
	for any $m \in \Z_{>0}$.
	Hence we have that (2) implies (3).
\end{proof}

\subsection{Exceptional loci}


\begin{dfn} \label{def:exceptional-locus}
	Let $f:X \to S$ be a proper morphism of noetherian schemes
	and let $L$ be an invertible sheaf on $X$.
	The {\em $f$-exceptional locus} $\E_f(L)$ of $L$ is 
	the set-theoretic union of all the integral closed subschemes $V$ of $X$
	such that $L|_V$ are not $f|_V$-big.
\end{dfn}

\begin{rem} \label{rem:exceptional-locus}
	By Theorem~\ref{WB-wb=b},
	the above definition 
	coincides with 
	the definition of $f$-exceptional locus in \cite{CT},
	which is the set-theoretic union of all the reduced closed subschemes $V$ of $X$
	such that $L|_V$ are not $f|_V$-weakly big.
\end{rem}

The proof of (\ref{enu:E=E|_Z-1}), (\ref{enu:E=X}), and (\ref{enu:E-closed}) of the following is the same as in \cite[Proposition~2.18]{CT}

\begin{prop}\label{prop:E-property}
	Let $f:X \to S$ be a projective morphism  
	of noetherian schemes 
	and let $L$ be an invertible sheaf on $X$. 
	\begin{enumerate}
		\item \label{enu:E-affineness}
			Let $S'$ be an open subscheme of $S$,
			let $X' \coloneqq X \times_S S'$, and
			$g:X' \to S'$ be the induced morphism.
			Then, we have
			\[
			\E_f(L) \cap X' = \E_g(L|_{X'}).
			\]
		\item \label{enu:E_red}
			$\E_f(L)=\E_{g}(L|_{X_\red})$ holds,
			where $g \coloneqq f|_{X_\red}$.
		\item\label{enu:E=E|_Z-1} 
			Let $m \in \Z_{>0}$, 
			and let $A$ be an $f$-ample invertible sheaf on $X$.
			Let $s \in H^0(X_\red,(L^{\otimes m}\otimes A^{-1})|_{X_\red})$,
			and let $Z \coloneqq Z(s)$.
			Then, $\E_f(L)=\E_{f|_Z}(L|_Z)$ holds.
		\item\label{enu:E=E|_Z-2} 
			Let $m \in \Z_{>0}$, 
			and let $A$ be an $f$-ample invertible sheaf on $X$.
			Let $s \in H^0(X,L^{\otimes m}\otimes A^{-1})$,
			and let $Z \coloneqq Z(s)$.
			Then, $\E_f(L)=\E_{f|_Z}(L|_Z)$ holds.
		\item \label{enu:E=X}
			$\E_f(L)=\Supp(X)$ if and only if $L$ is not $f$-weakly big.
		\item \label{enu:E-closed}
			$\E_f(L)$ is a closed subset of $X$.
		\item \label{enu:E-property-irr_comp}
			Let $V$ be an irreducible component of $\E_f(L)$
			equipped with the reduced scheme structure.
			Then $L|_V$ is not $f|_V$-big.
\end{enumerate}
\end{prop}

\begin{proof}
	(\ref{enu:E-affineness}) 
	It is enough to show that
	for any integral closed subscheme $V$ of $X$ 
	such that $V' \coloneqq V \times_X X' \neq \emptyset$,
	$L|_V$ is $f |_V$-big
	if and only if 
	$L|_{V'}$ is $g|_{V'}$-big.
	Note that $V' = V \times_S S'$.
	Thus the assertion follows from 
	Proposition~\ref{prop:WB-affineness} and Proposition~\ref{WB-affineness2}.
	
	(\ref{enu:E_red}) 
	$\E_f(L) \supseteq \E_{g}(L|_{X_\red})$ is clear.
	We show the opposite inclusion.
	Let $V$ be an integral closed subscheme of $X$
	such that $L|_V$ is not $f|_V$-big.
	Since $V$ is a closed subscheme of $X_\red$,
	$V \subseteq \E_{g}(L|_{X_\red})$ holds.
	Thus, we have $\E_f(L) \subseteq \E_{g}(L|_{X_\red})$.
	
	(\ref{enu:E=E|_Z-1}) 
	$\E_f(L) \supseteq \E_{f|_Z}(L|_Z)$ is clear.
	We show the opposite inclusion.
	Let $V$ be an integral closed subscheme of $X$
	such that 
	$L|_V$ is not $f|_V$-big.
	Then $s|_V \in H^0(V,(L^{\otimes m}\otimes A^{-1}) |_V)$ is equal to zero,
	which implies $\Supp(V) \subseteq \Supp(Z)$ 
	by Proposition~\ref{prop:ZL-3} and Proposition~\ref{prop:ZL-5}.
	Since $V$ is reduced, 
	$V$ is a closed subscheme of $Z$,
	and hence $V \subseteq \E_{f|_Z}(L|_Z)$ holds.
	Thus, we have $\E_f(L) \subseteq \E_{f|_Z}(L|_Z)$.
	
	(\ref{enu:E=E|_Z-2})
	The proof is the same as in (\ref{enu:E=E|_Z-1}).
	
	(\ref{enu:E=X}) 
	If $L$ is not $f$-weakly big,
	then for any irreducible component $Y$ of $X_\red$,
	$\Supp(Y) \subseteq \E_f(L)$ holds
	by Proposition~\ref{WB-wb=b}.
	Thus we have $\E_f(L) = \Supp(X)$.
	On the other hand, 
	if $L$ is $f$-weakly big,
	then $\E_f(L) \subseteq \Supp(Z(s))$ holds,
	where $s$ is a non-zero element in $H^0(X_\red, (L^{\otimes m} \otimes_{\MO_X} A^{-1})|_{X_\red})$
	and $A$ is an $f$-ample invertible sheaf on $X$.
	Thus we have $\E_f(L) \neq \Supp(X)$.
	
	(\ref{enu:E-closed}) 
	By (\ref{enu:E-affineness}),
	we may assume that $S$ is affine.
	By (\ref{enu:E_red}),
	we may assume that $X$ is reduced.
	If $L$ is not $f$-weakly big,
	by (\ref{enu:E=X}),
	we have $\E_f(L) = X$,
	which is closed.
	Thus, we may assume that $L$ is $f$-weakly big.
	Then, there is a nonzero element 
	$s \in H^0(X, L^{\otimes m} \otimes_{\MO_X} A^{-1})$
	for some $m \in \Z_{>0}$
	and $f$-ample invertible sheaf $A$ on $X$.
	Let $Z \coloneqq Z(s)$.
	By Proposition~\ref{prop:ZL-3},
	we have $\Supp(Z) \subsetneq \Supp(X)$.
	By (\ref{enu:E=E|_Z-2}),
	$\E_f(L) = \E_{f|_Z}(L|_Z)$ holds.
	By noetherian induction,
	we may assume that
	$\E_{f|_Z}(L|_Z)$ is a closed subset of $Z$.
	Hence it is also a closed subset of $X$.
	Thus, (\ref{enu:E-closed}) holds.
	
	(\ref{enu:E-property-irr_comp}) 
	Assume that $V$ is $f|_V$-big.
	Let $A$ be an $f|_V$-ample invertible sheaf on $V$.
	Let 
	$s \in H^0(V,(L|_V)^{\otimes m}\otimes A^{-1})$ 
	be a nonzero element
	with $m \in \Z_{>0}$,
	and let $Z \coloneqq Z(s)$.
	Then we have a set-theoretic inclusion
	\[
	\E_{f|_V}(L|_V)
	= \E_{f|_Z}(L|_Z)
	\subseteq Z,
	\]
	where the first equality follows from (\ref{enu:E=E|_Z-2}).
	Let $V'$ be the union of the other irreducible components of $\E_f(L)$ equipped with the reduced scheme structure.
	Let $v \in V \setminus V'$ be a closed point.
	Since $v \in V \subseteq \E_f(L)$, 
	there is an integral closed subscheme $W$ of $X$
	such that $v \in W$ and $L|_W$ is not $f|_W$-big.
	In other words, we have
	\[
	v \in W \subseteq \E_f(L) = V \cup V'.
	\]
	Since $v \notin V'$, 
	we have $W \nsubseteq V'$.
	Since $W$ is irreducible, 
	we have $W \subseteq V$.
	Since $(L|_V)|_W = L|_W$ is not $f|_W$-big,
	we have $W \subseteq \E_{f|_V}(L|_V)$.
	Since $v \in W$,
	we have $v \in \E_{f|_V}(L|_V)$.
	Thus, 
	$V \setminus V' \subseteq \E_{f|_V}(L|_V)$ holds.
	Since $\E_{f|_V}(L|_V)$ is a closed subset of $V$ 
	by (\ref{enu:E-closed}),
	$\overline{V \setminus V'} \subseteq \E_{f|_V}(L|_V)$
	holds,
	where $\overline{V \setminus V'}$ is the closure of $V \setminus V'$ in $V$.
	Thus we have a set-theoretic inclusion
	\[
	V 
	= \overline{V \setminus V'}
	\subseteq \E_{f|_V}(L|_V) 
	\subseteq Z.
	\]
	Hence 
	$\Supp(V) = \Supp(Z)$ holds.
	On the other hand,
	since $V$ is reduced and $s \neq 0$,
	we have $\Supp(Z) \neq \Supp(V)$
	by Proposition~\ref{prop:ZL-3}.
	These two estimates imply a contradiction.
\end{proof}

\section{Keel's theorem}\label{sec:K}


The purpose of this section is 
to show the relative versions of 
Nakamaye's theorem (see Theorem~\ref{K-main1})
and
Keel's result (see Theorem~\ref{K-main2}) for noetherian schemes.

In this section, we use the following notation.
%
%
Let $X$ be a scheme, and 
let $\sequence{Y}{,}{n}$, and $Y_\lambda$ ($\lambda \in \Lambda$) 
be closed subschemes of $X$.
Then, 
$\bigcup_{i =1}^n Y_i$ denotes 
the scheme-theoretic union of $\sequence{Y}{,}{n}$.
$\bigcap_{\lambda \in \Lambda} Y_\lambda$ denotes 
the scheme-theoretic intersection of $Y_\lambda$ ($\lambda \in \Lambda$).
$Y_1 \subseteq Y_2$ denotes that
$Y_1$ is a closed subscheme of $Y_2$.
To avoid confusion,
we use the notation $\Supp(Y)$
when a closed subscheme $Y$ of $X$ is considered as a closed subset of $X$.

\begin{lem}\label{lem:K-1}
	Let $f:X \to S$ be a projective morphism of noetherian schemes,
	with $S$ irreducible.
	Let $d \coloneqq \dim(X/S)$.
	Let $\MF$ be a coherent sheaf on $X$, and
	let $L$ be an $f$-nef invertible sheaf on $X$.
	Then, for any $j \in \Z_{\ge0}$, 
	there is a positive real number $C \in \R_{>0}$
	such that for sufficiently large $m \in \Z_{>0}$,
	\[
	\glen_S\,R^j f_* (\MF \otimes_{\MO_X} L ^{\otimes m})
	\le C m^{d-j}
	\]
	holds.
\end{lem}

\begin{proof}
	We may assume that $S$ is affine 
	and written as $S = \Spec\,R$
	with a noetherian ring $R$.
	Let $R' \coloneqq R/ \sqrt{0}$, and 
	let $\sigma$ be the generic point of $S$.
	Let $\MN$ be the nilradical ideal sheaf on $X$.
	First, we show the case where $\MF = \MN^i \otimes_{\MO_X} \ME$ 
	for a coherent locally free sheaf $\ME$ on $X$ and
	$i \in \Z_{\ge0}$.
	Since $X$ is noetherian,
	$\MN^{l+1} = 0$ holds for some $l \in \Z_{>0}$.
	Hence the assertion clearly holds if $i \ge l+1$.
	Let $M_m = \ME \otimes_{\MO_X} L ^{\otimes m}$.
	For every $i \in \{0,1,\dots,l\}$,
	the exact sequence
	\[
	0
	\to \MN^{i+1}
	\to \MN^i
	\to \MN^i/\MN^{i+1}
	\to 0
	\]
	induces the exact sequence
	\[
 	H^j(X,\MN^{i+1} \otimes_{\MO_X} M_m)
	\to H^j(X,\MN^i \otimes_{\MO_X} M_m)
	\to H^j(X,(\MN^i/\MN^{i+1}) \otimes_{\MO_X} M_m).
	\]
	By descending induction on $i$,
	it is enough to show the case 
	where $\MF = (\MN^i/\MN^{i+1}) \otimes_{\MO_X} \ME$
	for $i \in \{0,1,\dots,l\}$.
	Since $\MN^i/\MN^{i+1}$ is $(\MO_X/\MN)$-module,
	we have
	\[
	H^j(X,(\MN^i/\MN^{i+1}) \otimes_{\MO_X} M_m) 
	= H^j(X_\red,((\MN^i/\MN^{i+1}) \otimes_{\MO_X} M_m)|_{X_\red}).
	\]
	Since $f(X_\red)$ is a closed subscheme of $S_\red$,
	this cohomology is also an $R'$-module.
	Thus, 
	there is a positive real number $C \in \R_{>0}$
	such that
	for sufficiently large $m \in \Z_{>0}$,
	\begin{align*}
			&\glen_R\,H^j(X,(\MN^i/\MN^{i+1}) \otimes_{\MO_X} M_m) \\
	=		&\glen_{R'}\,H^j(X_\red,((\MN^i/\MN^{i+1}) \otimes_{\MO_X} M_m)|_{X_\red}) \\
	=		&\dim_{K(R')} H^j(X_{\red,\sigma},
				((\MN^i/\MN^{i+1}) \otimes_{\MO_X} M_m)|_{X_{\red,\sigma}}) \\
	\le 	&C m^{d-j}
	\end{align*}
	holds.
	The first and second equalities follow from Remark~\ref{rem:GR-1},
	and
	the last inequality follows from \cite[Proposition~4.4~(2)]{Bir17}.
	Thus we have shown the case 
	where $\MF = \MN^i \otimes_{\MO_X} \ME$ for $i \in \Z_{\ge0}$.
	In particular, the assertion holds if $\MF$ is locally free.
	
	Second, we show the general case.
	We can write $\MF$ as a quotiont of a coherent locally free sheaf $\ME$.
	Let $\MK$ be the kernel sheaf of the natural sheaf homomorphism $\ME \to \MF$.
	Note that also $\MK$ is coherent.
	We obtain the exact sequence 
	\[
	0
	\to \MK
	\to \ME
	\to \MF
	\to 0.
	\] 
	This induces the exact sequence 
	\[
	H^j(X, \ME \otimes_{\MO_X} L^{\otimes m})
	\to H^j(X, \MF \otimes_{\MO_X} L^{\otimes m})
	\to H^{j+1}(X, \MK \otimes_{\MO_X} L^{\otimes m}).
	\]
	By descending induction on $j$,
	there is a positive real number $C_1 \in \R_{>0}$
	such that
	\[
	\glen_R\,  H^{j+1}(X, \MK \otimes_{\MO_X} L^{\otimes m})
	\le C_1 m^{d-j-1}
	\]
	holds
	for sufficiently large $m \in \Z_{>0}$.
	On the other hand,
	by the case where $\MF$ is locally free,
	there is a positive real number $C_2 \in \R_{>0}$
	such that
	\[
	\glen_R\, H^j(X, \ME \otimes_{\MO_X} L^{\otimes m})
	\le C_2 m^{d-j}
	\]
	holds
	for sufficiently large $m \in \Z_{>0}$.
	Hence we have
	\begin{align*}
			&\glen_R\, H^j(X, \MF \otimes_{\MO_X} L^{\otimes m}) \\
	\le 	&\glen_R\, H^{j+1}(X, \MK \otimes_{\MO_X} L^{\otimes m})
				+ \glen_R\, H^j(X, \ME \otimes_{\MO_X} L^{\otimes m}) \\
	\le 	&C_1 m^{d-j-1} + C_2 m^{d-j} \\
	\le 	&(C_1 + C_2) m^{d-j}
	\end{align*}
	for sufficiently large $m \in \Z_{>0}$.
\end{proof}

\begin{lem}\label{lem:K-3}
	Let $f:X \to S$ be a projective surjective morphism 
	of irreducible noetherian schemes.
	Let $d \coloneqq \dim(X/S)$.
	Let $\ME$ be a coherent locally free sheaf on $X$, and
	let $L$ be an $f$-nef invertible sheaf on $X$.
	If $L$ is $f$-big,
	there is a positive real number $C \in \R_{>0}$
	such that for sufficiently large $m \in \Z_{>0}$,
	\[
	\glen_S\, f_* (\ME \otimes_{\MO_X} L ^{\otimes m})
	\ge C m^{d}
	\]
	holds.
\end{lem}

\begin{proof}
	By Proposition~\ref{prop:WB-affineness},
	there is a non-empty open affine subscheme $S'$ of $S$
	such that 
	$L|_{X'}$ is $g$-weakly big,
	where $X' \coloneqq X \times_S S'$ 
	and $g:X' \to S'$ is the natural morphism.
	Thus, by replacing $f:X \to S$ with $g:X' \to S'$,
	we may assume that 
	$S$ is affine and written as $S = \Spec\, R$
	with a noetherian ring $R$.
	Let $\MN$ be the nilradical ideal sheaf of $X$, and
	let $R' \coloneqq R/\sqrt{0}$.
	Let $M_m \coloneqq \ME \otimes L^{\otimes m}$.
	The exact sequence 
	\[
 	0
	\to \MN
	\to \MO_X
	\to \MO_{X_\red}
	\to 0
	\]
	induces the exact sequence
	\[
	H^0(X,M_m)
	\to H^0(X_\red,M_m|_{X_\red})
	\to H^{1}(X,\MN \otimes M_m).
	\]
	By Lemma~\ref{lem:K-1},
	there is a positive real number $C_1 \in \R_{>0}$
	such that
	for sufficiently large $m \in \Z_{>0}$,
	$H^{1}(X,\MN \otimes M_m) \le C_1 m^{d-1}$
	holds.
	Let $\sigma$ be the generic point of $S$.
	By Proposition~\ref{WB-lower bound},
	$L|_{X_{\red,\sigma}}$ is big over $\Spec\, K(R')$.
	Thus, 
	there is a positive real number $C_2 \in \R_{>0}$
	such that
	for sufficiently large $m \in \Z_{>0}$,
	\begin{align*}
			\glen_R\, H^0(X_\red,M_m|_{X_\red}) 
	=		&\glen_{R'}\, H^0(X_\red,M_m|_{X_\red}) \\
	=		&\dim_{K(R')} H^0(X_{\red,\sigma},
				M_m|_{X_{\red,\sigma}}) \\
	\ge 	&C_2 m^{d}
	\end{align*}
	holds.
	The first and second equalities follow from Remark~\ref{rem:GR-1},
	and the last inequality follows from \cite[Proposition~4.3]{Bir17}.
	Thus, for sufficiently large $m \in \Z_{>0}$, 
	we have
	\begin{align*}
			&\glen_R\, H^0(X,M_m) \\
	\ge 	&\glen_R\, H^0(X_\red,M_m|_{X_\red})
			 - \glen_R\, H^{1}(X,\MN \otimes_{\MO_X} M_m) \\
	\ge 	&C_2 m^{d} - C_1 m^{d-1} \\
	\ge 	&C m^{d},
	\end{align*}
	where $C \coloneqq C_2 / 2$.
\end{proof}

\begin{lem}\label{lem:K-4}
	Let $f:X \to S$ be a projective morphism of noetherian schemes,
	with $S$ affine.
	Let $L$ be an $f$-nef invertible sheaf on $X$, and
	let $A$ be an $f$-ample invertible sheaf on $X$.
	Assume that 
	$L$ is $f$-weakly big.
	Then, for sufficiently large $m \in \Z_{>0}$,
	there is a section $s \in H^0(X,L^{\otimes m}\otimes_{\MO_X} A^{-1})$
	such that 
	$s|_{X_\red} \neq 0$ holds.
\end{lem}

\begin{proof}
	Since $L$ is $f$-weakly big,
	$X \neq \emptyset$ holds.
	By Theorem~\ref{WB-wb=b},
	there is an irreducible component $V$ of $X_\red$
	such that $L|_V$ is $f$-big.
	Let $X = \bigcup_{i=1}^n X_i$ be an irredundant decomposition of $X$
	(cf. Definition-Proposition~\ref{Dec-scheme}).
	Then, for some $j \in \{1,2,\dots,n\}$,
	$V = X_{j,\red}$ holds.
	In particular, $L|_{X_j}$ is $f|_{X_j}$-big.
	After permuting the indices,
	we may assume $j=1$.
	Let $\MN$ be the nilradical ideal sheaf of $X_1$.
	Let $M_m \coloneqq L^{\otimes m}\otimes A^{-1}$.
	The exact sequence 
	\[
	0
	\to \MN
	\to \MO_{X_1}
	\to \MO_{X_{1,\red}}
	\to 0
	\]
	induces the exact sequence 
	\[
	H^0(X_1, M_m |_{X_1})
	\stackrel{\phi_m}{\longrightarrow} 
	H^0(X_{1,\red}, M_m |_{X_{1,\red}})
	\stackrel{\phi'_m}{\longrightarrow} 
	H^1(X_1, \MN \otimes_{\MO_{X_1}} M_m |_{X_{1}})
	\]
	Let $d \coloneqq \dim(X_1/f(X_1))$.
	Note that $d \ge 0$.
	We write $f(X_1) = \Spec\, R$ with a noetherian ring $R$.
	By Lemma~\ref{lem:K-1},
	there is a positive real number $C_1 \in \R_{>0}$
	such that
	for sufficiently large $m \in \Z_{>0}$,
	$\glen_R\, H^1(X_1, \MN \otimes_{\MO_{X_1}} M_m |_{X_{1}}) \le C_1 m^{d-1}$ holds.
	By Lemma~\ref{lem:K-3},
	there is a positive real number $C_2 \in \R_{>0}$
	such that
	for sufficiently large $m \in \Z_{>0}$,
	$\glen_R\, H^0(X_{1,\red}, M_m |_{X_{1,\red}}) \ge C_2 m^{d}$ holds.
	Thus we have 
	\[
	\glen_R\, \Im\, \phi_m
	= \glen_R\, \Ker\, \phi'_m
	\ge C_2 m^d - C_1 m^{d-1}.
	\]
	Let $X' \coloneqq \bigcup_{i=2}^n X_i$, and
	let $T \coloneqq X_1 \cap X'$.
	The Mayer--Vietoris sequence
	\[
	0
	\to \MO_X
	\to \MO_{X_1} \oplus \MO_{X'}
	\to \MO_T
	\to 0
	\]
	induces the exact sequence
	\[
	0
	\to H^0(X,M_m)
	\to H^0(X_1,M_m|_{X_1}) \oplus H^0(X',M_m|_{X'})
	\to H^0(T,M_m|_T).
	\]
	Let $\psi_m:H^0(X,M_m) \to H^0(X_1,M_m|_{X_1})$ and 
	$\psi'_m:H^0(X_1,M_m|_{X_1}) \to H^0(T,M_m|_T)$ be 
	the natural homomorphisms.
	Then,
	from the above exact sequence,
	we have 
	$\Im\, \psi_m \supseteq \Ker\, \psi'_m$.
	By Lemma~\ref{lem:K-1},
	there is a positive real number $C_3 \in \R_{>0}$
	such that
	for sufficiently large $m \in \Z_{>0}$,
	$\glen_R\, H^0(T,M_m|_T) \le C_3 m^{d-1}$ holds.
	Assume that 
	$\phi_m \circ \psi_m = 0$,
	i.e.,
	$\Im\, \psi_m \subseteq \Ker\, \phi_m$.
	Then 
	\begin{align*}
	C_2 m^d - C_1 m^{d-1}
	&\le
	\glen_R\, \Im\, \phi_m \\
	&= 
	\glen_R\, H^0(X_1, M_n |_{X_1}) - \glen_R\, \Ker\, \phi_m \\
	&\le
	\glen_R\, H^0(X_1, M_n |_{X_1}) - \glen_R\, \Im\, \psi_m \\
	&\le 
	\glen_R\, H^0(X_1, M_n |_{X_1}) - \glen_R\, \Ker\, \psi'_m \\
	&=
	\glen_R\, \Im\, \psi'_m \\
	&\le 
	C_3 m^{d-1}
	\end{align*}
	holds.
	Thus we have 
	$\phi_m \circ \psi_m \neq 0$,
	which implies the assertion.
\end{proof}

\begin{lem}\label{lem:K-6}
	Let $X$ be a noetherian scheme,
	and let $V$ and $W$ be closed subschemes of $X$.
	Let $\MI$ (resp. $\MJ$) be the ideal sheaf of $V$ (resp. $W$).
	If $\Supp(V) \subseteq \Supp(W)$ holds,
	then $V \subseteq W^r$ holds
	for sufficiently large $r \in \Z_{>0}$,
	where $W^r$ is the closed subscheme of $X$ 
	determined by the coherent ideal sheaf $\MJ^r$ on $X$.
\end{lem}

\begin{proof}
	Since $X$ is noetherian, we may assume that 
	$X$ is affine 
	and written as $X = \Spec\,R$ with a noetherian ring $R$.
	Let $I \coloneqq H^0(X, \MI)$ and $J \coloneqq H^0(X, \MJ)$.
	It is enough to show that
	$I \supseteq J^r$ for some $r \in \Z_{>0}$.
	Since $R$ is noetherian,
	$J$ is generated by a finite number of elements $\sequence{f}{,}{l}$
	as an ideal of $R$.
	Since $\Supp(V) \subseteq \Supp(W)$,
	we have $\sqrt{I} \supseteq \sqrt{J} \supseteq J$.
	Let $r_0$ be a positive integer such that
	$f_i ^{r_0} \in I$ holds for any $i \in \{1,2,\dots,l\}$.
	Then, by the pigeonhole principle, we have 
	\[
	J^{l(r_0 -1) +1} 
	\subseteq (\sequence{f^{r_0}}{,}{l}) 
	\subseteq I.
	\]
	
\end{proof}

\begin{prop}\label{prop:K-E subseteq B_+}
	Let $f:X \to S$ be a projective morphism 
	of noetherian schemes.
	Let $L$ be an $f$-nef invertible sheaf on $X$.
	Then, we have $\E_f(L) \subseteq \B_{+,f}(L)$. 
\end{prop}

\begin{proof}
	By Proposition~\ref{BL-affineness}~(\ref{enu:BL-affineness-3})
	and Proposition~\ref{prop:E-property}~(\ref{enu:E-affineness}),
	we may assume that $S$ is affine. 
	Let $A$ be an $f$-ample invertible sheaf on $X$.
	Let $V$ be an integral closed subscheme of $X$.
	Assume that $\Supp(V) \nsubseteq \B_{+,f}(L)$.
	Then, there are positive integers $m, r \in \Z_{>0}$ and 
	a section $s \in H^0(X, L^{\otimes mr}\otimes_{\MO_X} A^{-r})$
	such that $\Supp(V) \nsubseteq \Supp(Z(s))$ holds.
	By Proposition~\ref{prop:ZL-3} and Proposition~\ref{prop:ZL-5},
	we have 
	$s|_V \neq 0$.
	Thus $L|_V$ is $f|_V$-big.
	This induces that 
	if $L|_V$ is not $f|_V$-big,
	then $\Supp(V) \subseteq \B_{+,f}(L)$ holds.
	Hence we have $\E_f(L) \subseteq \B_{+,f}(L)$.
\end{proof}

\begin{thm}\label{K-main1}
	Let $f:X \to S$ be a projective morphism 
	of noetherian schemes.
	Let $L$ be an $f$-nef invertible sheaf on $X$.
	Then, we have $\B_{+,f}(L)=\E_f(L)$. 
\end{thm}

\begin{proof}
	By Proposition~\ref{BL-affineness}~(\ref{enu:BL-affineness-3}) and
	Proposition~\ref{prop:E-property}~(\ref{enu:E-affineness}), 
	we may assume that $S$ is affine.
	Let $A$ be an $f$-very ample invertible sheaf on $X$.
	By Proposition~\ref{prop:K-E subseteq B_+}, 
	$\E_f(L) \subseteq \B_{+,f}(L)$ holds.
	Assume that $L$ is not $f$-weakly big.
	Then we have $\Supp(X) = \E_f(L)$
	by Proposition~\ref{prop:E-property}~(\ref{enu:E=X}).
	Hence we have
	\[
	\Supp(X) = \E_f(L) \subseteq \B_{+,f}(L) \subseteq \Supp(X),
	\]
	which implies
	$\E_f(L) = \B_{+,f}(L)$.
	Thus, we may assume that $L$ is $f$-weakly big.
	Then, by Lemma~\ref{lem:K-4},
	there are $m \in \Z_{>0}$ 
	and $s \in H^0(X,L^{\otimes m}\otimes A^{-1})$
	such that $s|_{X_\red} \neq 0$ holds.
	Let $X \coloneqq \bigcup_{j=1}^n X_j$ be an irredundant decomposition of $X$
	(cf. Definition-Proposition~\ref{Dec-scheme}).
	Let 
	\[
	J' \coloneqq \{ j \in \{1,2,\dots,n\} \mid \Supp(X_j) \nsubseteq \Supp(Z(s))\},
	\]
	and let
	\[
	J'' \coloneqq \{ j \in \{1,2,\dots,n\} \mid \Supp(X_j) \subseteq \Supp(Z(s))\}.
	\]
	Let 
	$X' \coloneqq \bigcup_{j \in J'} X_j$, and 
	let $X'' \coloneqq \bigcup_{j \in J''} X_j$.
	By Lemma~\ref{lem:K-6} and Proposition~\ref{prop:ZL-4}, 
	$X'' \subseteq Z(s^{\otimes r})$ holds
	for sufficiently large $r \in \Z_{>0}$.
	By Fujita's vanishing theorem (\cite[Theorem~1.5]{Kee:ample-filter}), 
	there is a positive integer $r_0 \in \Z_{>0}$
	such that
	for any integer $r \ge r_0$ and 
	for any $f$-nef invertible sheaf $L'$ on $X'$,
	\[
	H^1(X',(A|_{X'})^{\otimes r-1} \otimes_{\MO_{X'}} L') = 0
	\] 
	holds.
	Pick $r \in \Z_{>0}$ such that 
	these two conditions hold.
	Let $Z \coloneqq Z(s^{\otimes r})$.
	Let $Z' \coloneqq Z(s^{\otimes r}|_{X'}) = Z \cap X'$,
	where the second equality follows from Proposition~\ref{prop:ZL-5}. 
	Let $\MI_{Z'}$ be the ideal sheaf of $Z'$ on $X'$.
	By Definition-Proposition~\ref{prop:ZL-effective divisor}
	and the definition of $X'$, 
	$Z(s|_{X'})$ defines an effective Cartier divisor on $X'$. 
	Since 
	$\Supp(Z') = \Supp(Z(s|_{X'}))$,
	also $Z'$ defines an effective Cartier divisor on $X'$.
	Thus 
	we have 
	$\MI_{Z'} \simeq 
	(L|_{X'})^{-mr} \otimes_{\MO_{X'}} (A|_{X'})^{\otimes r}$
	by Definition-Proposition~\ref{prop:ZL-effective divisor}.
	We now show the following claim.
	
	\begin{claim}\label{claim:K-1}
		$\B_{+,f|_Z}(L|_Z) = \B_{+,f}(L)$ holds.
	\end{claim}
	
	\begin{proof}[Proof of Claim~\ref{claim:K-1}]
	By Proposition~\ref{prop:BL-|_Y}~(\ref{enu:BL-|_Y-B_+})
	and the definition of $\B_{+,f}(L)$,
	we have 
	\[
	\B_{+,f|_Z}(L|_Z)
	\subseteq \B_{+,f}(L)
	\subseteq \Supp(Z).
	\]
	Let $x \in \Supp(Z) \setminus \B_{+,f|_Z}(L|_Z)$.
	It is enough to show that 
	$x \notin \B_{+,f}(L)$.
	Let $m'$ and $ r'$ be positive integers
	such that 
	\[
	\B_{+,f|_Z}(L|_Z) = \Bs_{f|_Z}((L^{\otimes m'r'} \otimes_{\MO_X} A^{-r'})|_Z)
	\] 
	and
	$m'r' \ge mr$ hold.
	Then,
	there is a section $t_1 \in H^0(Z, (L^{\otimes m'r'}\otimes_{\MO_X} A^{-r'})|_Z )$
	such that $t_1(x) \neq 0$.
	Since $A$ is $f$-very ample,
	there is a section $t_2 \in H^0(Z, A^{r'-1})$
	such that $t_2(x) \neq 0$.
	Let $M \coloneqq L^{\otimes m'r'}\otimes_{\MO_X} A^{-1}$ and
	let $t \coloneqq t_1 \otimes t_2 \in H^0(Z,M|_Z)$.
	Then $t(x) \neq 0$ holds.
	The exact sequence 
	\[
	0
	\to \MI_{Z'}
	\to \MO_{X'}
	\to \MO_{Z'}
	\to 0
	\]
	induces the exact sequence
	\[
	H^0(X', M|_{X'})
	\to H^0(Z', M|_{Z'})
	\to H^1(X', \MI_{Z'} \otimes_{\MO_{X'}} M|_{X'}).
	\]
	Since 
	\[
	\MI_{Z'} \otimes_{\MO_{X'}} M|_{X'}
	= (L|_{X'})^{\otimes m'r'-mr} \otimes_{\MO_{X'}} (A|_{X'})^{\otimes r-1}
	\]
	and $(L|_{X'})^{\otimes m'r'-mr}$ is $f|_{X'}$-nef,
	$H^1(X', \MI_{Z'} \otimes_{\MO_{X'}} M|_{X'}) = 0$ holds
	by the definition of $r$.
	Thus, there is a section $t' \in H^0(X', M|_{X'})$
	such that $t'|_{Z'} = t|_{Z'}$.
	By the Mayer--Vietoris sequence,
	we have the exact sequence
	\[
	0
	\to H^0(X,M)
	\stackrel{\phi}{\to} H^0(X',M|_{X'}) \oplus H^0(X'',M|_{X''})
	\stackrel{\psi}{\to} H^0(T, M|_T),
	\]
	where $T \coloneqq X' \cap X''$.
	Note that
	$T
	\subseteq Z'
	\subseteq X'$
	and
	$T 
	\subseteq X'' 
	\subseteq Z$
	since $T = X' \cap X''$, $Z' = X' \cap Z$, and $X'' \subseteq Z$.
	We have 
	\[
	t'|_T
	= (t'|_{Z'})|_T
	= (t|_{Z'})|_T
	= t|_T
	= (t|_{X''})|_T,
	\]
	i.e., 
	$\psi(t', t|_{X''}) = 0$.
	Thus,
	there is a section $t'' \in H^0(X, M)$
	such that $\phi(t'') = (t', t|_{X''})$.
	%
%
%
%
%
	Note that
	$t''|_{Z'} = t'|_{Z'} = t|_{Z'}$
	and
	$t''|_{X''} = t|_{X''}$.
	Since $X = X' \cup X''$,
	we have $\Supp(X) = \Supp(X') \cup \Supp(X'')$.
	If $x \in \Supp(X')$,
	we have 
	\[
	t''(x) = (t''|_{Z'})(x) = (t|_{Z'})(x) = t(x) \neq 0,
	\]
	since $x \in \Supp(X') \cap \Supp(Z) = \Supp(Z')$.
	On the other hand,
	if $x \in \Supp(X'')$,
	we have
	\[
	t''(x) =(t''|_{X''})(x) = (t|_{X''})(x) = t(x) \neq 0.
	\]
	In any case, we have $t''(x) \neq 0$,
	which implies
	$x \notin \Bs_f(L^{\otimes m'r'}\otimes_{\MO_X} A^{-1}) 
	\supseteq \B_{+,f}(L)$.
	Hence the claim holds.
	\end{proof}
	
	Since $s|_{X_\red} \neq 0$,
	$\Supp(Z) \subsetneq \Supp(X)$ holds by Proposition~\ref{prop:ZL-3}.
	By noetherian induction,
	it holds that
	$\E_{f |_Z}(L |_Z) = \B_{+,f |_Z}(L |_Z)$.
	By Claim~\ref{claim:K-1},
	we have 
	$\B_{+,f|_Z}(L|_Z) = \B_{+,f}(L)$.
	By Proposition~\ref{prop:E-property}~(\ref{enu:E=E|_Z-2}), 
	we have 
	$\E_f(L) = \E_{f |_Z}(L |_Z)$.
	Thus, we have 
	$\B_{+,f}(L) = \E_f(L)$.
\end{proof}

\begin{thm}\label{K-main2}
	Let $f:X \to S$ be a projective morphism 
	of noetherian schemes 
	and let $L$ be an $f$-nef invertible sheaf on $X$.
	Then, there exists a closed subscheme $F$ of $X$
	such that
	$\Supp(F) = \E_f(L)$
	and 
	$\SB_f(L) = \SB_{f|_F}(L|_F)$ hold.
\end{thm}

The proof is almost the same as in Theorem~\ref{K-main1}.

\begin{proof}
	By Proposition~\ref{BL-affineness}~(\ref{enu:BL-affineness-3}) and
	Proposition~\ref{prop:E-property}~(\ref{enu:E-affineness}), 
	We may assume that $S$ is affine.
	Let $A$ be an $f$-very ample invertible sheaf on $X$.
	If $L$ is not $f$-weakly big,
	we have $\E_f(L) = \Supp(X)$
	by Proposition~\ref{prop:E-property}~(\ref{enu:E=X}),
	and hence the assertion holds 
	for $F = X$.
	Thus, we may assume that $L$ is $f$-weakly big.
	%
	%
	Then, by Lemma~\ref{lem:K-4},
	there are $m \in \Z_{>0}$ and $s \in H^0(X,L^{\otimes m}\otimes A^{-1})$
	such that $s|_{X_\red} \neq 0$ holds.
	Let $X = \bigcup_{j=1}^n X_j$ be an irredundant decomposition of $X$
	(cf. Definition-Proposition~\ref{Dec-scheme}).
	Let 
	\[
	J' \coloneqq \{ j \in \{1,2,\dots,n\} \mid \Supp(X_j) \nsubseteq \Supp(Z(s))\},
	\]
	and let
	\[
	J'' \coloneqq \{ j \in \{1,2,\dots,n\} \mid \Supp(X_j) \subseteq \Supp(Z(s))\}.
	\]
	Let 
	$X' \coloneqq \bigcup_{j \in J'} X_j$, and 
	let $X'' \coloneqq \bigcup_{j \in J''} X_j$.
	By Lemma~\ref{lem:K-6} and Proposition~\ref{prop:ZL-4}, 
	$X'' \subseteq Z(s^{\otimes r})$ holds
	for sufficiently large $r \in \Z_{>0}$.
	By Fujita's vanishing theorem (cf. \cite[Theorem~1.5]{Kee:ample-filter}), 
	there is a positive integer $r_0 \in \Z_{>0}$
	such that
	for any integer $r \ge r_0$ and 
	for any $f$-nef invertible sheaf $L'$ on $X'$,
	$H^1(X',(A|_{X'})^{\otimes r-1} \otimes_{\MO_{X'}} L') = 0$ holds.
	Pick $r \in \Z_{>0}$ such that 
	these two conditions hold.
	Let $Z \coloneqq Z(s^{\otimes r})$.
	Let $Z' \coloneqq Z(s^{\otimes r}|_{X'}) = Z \cap X'$,
	where the second equality follows from Proposition~\ref{prop:ZL-5}. 
	Let $\MI_{Z'}$ be the ideal sheaf of $Z'$ on $X'$.
	By Definition-Proposition~\ref{prop:ZL-effective divisor}, 
	$Z(s|_{X'})$ defines an effective Cartier divisor on $X'$. 
	Since 
	$\Supp(Z') = \Supp(Z(s|_{X'}))$,
	also $Z'$ defines an effective Cartier divisor on $X'$.
	By Definition-Proposition~\ref{prop:ZL-effective divisor}, 
	we have 
	$\MI_{Z'} \simeq 
	(L|_{X'})^{-mr} \otimes_{\MO_{X'}} (A|_{X'})^{\otimes r}$.
	We now show the following claim.
	
	\begin{claim}\label{claim:K-2}
		$\SB_{f|_Z}(L|_Z) = \SB_f(L)$ holds.
	\end{claim}
	
	\begin{proof}[Proof of Claim~\ref{claim:K-2}]
	We have 
	\[
	\SB_{f|_Z}(L|_Z)
	\subseteq \SB_{f}(L)
	\subseteq \B_{+,f}(L)
	\subseteq \Supp(Z),
	\]
	where the first, second, and last inclusions respectively
	follow from Proposition~\ref{prop:BL-|_Y}~(\ref{enu:BL-|_Y-SB}), the definition of $\SB_f(L)$, and the definition of $\B_{+,f}(L)$.
	%
	Let $x \in \Supp(Z) \setminus \SB_{f|_Z}(L|_Z)$.
	It is enough to show that 
	$x \notin \SB_{f}(L)$.
	Let $m'$ be a positive integer
	such that 
	$
	\SB_{f|_Z}(L|_Z) = \Bs_{f|_Z}((L|_Z)^{\otimes m'} )
	$ 
	and
	$m' \ge mr$ hold.
	Then,
	there is a section $t \in H^0(Z, (L|_Z)^{\otimes m'} )$
	such that $t(x) \neq 0$.
	Let $M = L^{\otimes m'}$.
	The exact sequence 
	\[
	0
	\to \MI_{Z'}
	\to \MO_{X'}
	\to \MO_{Z'}
	\to 0
	\]
	induces the exact sequence
	\[
	H^0(X', M|_{X'})
	\to H^0(Z', M|_{Z'})
	\to H^1(X', \MI_{Z'} \otimes_{\MO_{X'}} M|_{X'}).
	\]
	Since 
	\[
	\MI_{Z'} \otimes_{\MO_{X'}} M|_{X'}
	= (L|_{X'})^{\otimes m'-mr} \otimes_{\MO_{X'}} (A|_{X'})^{\otimes r}
	\]
	and $(L|_{X'})^{\otimes m'-mr}$ is $f|_{X'}$-nef,
	$H^1(X', \MI_{Z'} \otimes_{\MO_{X'}} M|_{X'}) = 0$ holds
	by the definition of $r$.
	Thus, there is a section $t' \in H^0(X', M|_{X'})$
	such that $t'|_{Z'} = t|_{Z'}$.
	By the Mayer--Vietoris sequence,
	we have the exact sequence
	\[
	0
	\to H^0(X,M)
	\stackrel{\phi}{\to} H^0(X',M|_{X'}) \oplus H^0(X'',M|_{X''})
	\stackrel{\psi}{\to} H^0(T, M|_T),
	\]
	where $T \coloneqq X' \cap X''$.
	Note that
	$T
	\subseteq Z'
	\subseteq X'$
	and
	$T \subseteq X'' \subseteq Z$
	since $T = X' \cap X''$, $Z' = X' \cap Z$, and $X'' \subseteq Z$.
	We have 
	\[
	t'|_T
	= (t'|_{Z'})|_T
	= (t|_{Z'})|_T
	= t|_T
	= (t|_{X''})|_T,
	\]
	i.e., 
	$\psi(t', t|_{X''}) = 0$.
	Thus,
	there is a section $t'' \in H^0(X, M)$
	such that $\phi(t'') = (t', t|_{X''})$.
	Note that
	$t''|_{Z'} = t'|_{Z'} = t|_{Z'}$
	and
	$t''|_{X''} = t|_{X''}$.
	Since $X = X' \cup X''$,
	we have $\Supp(X) = \Supp(X') \cup \Supp(X'')$.
	If $x \in \Supp(X')$,
	we have 
	\[
	t''(x) = (t''|_{Z'})(x) = (t|_{Z'})(x) = t(x) \neq 0,
	\]
	since $x \in \Supp(X') \cap \Supp(Z) = \Supp(Z')$.
	On the other hand,
	if $x \in \Supp(X'')$,
	we have
	\[
	t''(x) =(t''|_{X''})(x) = (t|_{X''})(x) = t(x) \neq 0.
	\]
	In any case, we have $t''(x) \neq 0$,
	which implies 
	$x \notin \Bs_f(L^{\otimes m'}) \supseteq \SB_f(L)$.
	Hence the claim holds.
	\end{proof}
	
	Since $s|_{X_\red} \neq 0$,
	$\Supp(Z) \subsetneq \Supp(X)$ holds by Proposition~\ref{prop:ZL-3}.
	By noetherian induction,
	there is a closed subscheme $F'$ of $Z$ 
	such that $\Supp(F') = \E_{f |_Z}(L |_Z)$
	and $\SB_{f|_Z}(L|_Z) = \SB_{f|_{F'}}(L|_{F'})$ hold.
	Then we have
	$
	\E_f(L) = \E_{f |_Z}(L |_Z) = \Supp(F')
	$,
	where the first equality follows from Proposition~\ref{prop:E-property}~(\ref{enu:E=E|_Z-2}).
	We also have
	$
	\SB_f(L) = \SB_{f|_Z}(L|_Z) = \SB_{f|_{F'}}(L|_{F'})
	$,
	where the first equality follows from Claim~\ref{claim:K-2}.
	Hence the assertion holds.
\end{proof}

%
%

\section{Appendix: Kodaira's lemma} \label{sec:App}


In this section, 
we generalize Lemma~\ref{E-main}
(see Theorem~\ref{thm:App-main}).
This implies the relative version of Kodaira's lemma for 
reduced noetherian schemes (see Corollary~\ref{App-Kodaira's lemma}).

\begin{lem}\label{lem:App-1}
	Let $f:A \to B$ be an injective ring homomorphism 
	of noetherian integral domains, and 
	let $M$ be a finitely generated $A$-module.
	Then, we have 
	$\glen_A\,M = \glen_B (M \otimes_A B)$.
\end{lem}

\begin{proof}
	The assertion follows from the following equality:
	\begin{align*}
	\glen_A\,M 
	&= \dim_{K(A)} (M \otimes_A K(A)) \\
	&= \dim_{K(B)} (M \otimes_A K(A) \otimes_{K(A)} K(B)) \\
	&= \dim_{K(B)} (M \otimes_A K(B)) \\	
	&= \dim_{K(B)} (M \otimes_A B \otimes_B K(B)) \\
	&= \glen_B (M \otimes_A B),
	\end{align*}
	where the first and last equalities follows from Remark~\ref{rem:GR-1}.
\end{proof}

\begin{prop}\label{GR-4}
Let $A$ be an integral domain and 
let $M$ be a finitely generated $A$-module.
Let $r \in \Z_{\ge 0}$.
Then, the following conditions are equivalent:
\begin{enumerate}
\item 
	$r \le \glen_A\, M$;
\item
	There is an injective $A$-homomorphism
	$A^{\oplus r} \to M$.
\end{enumerate}
\end{prop}

\begin{proof} 
	Let $K \coloneqq K(A)$.
	First, we show that (1) implies (2).
	Since $\glen_A\, M = \dim_K (M \otimes_A K)$
	by Remark~\ref{rem:GR-1},
	there are elements $x_1, x_2, \dots , x_r \in M$
	such that 
	$x_1 \otimes_A 1, x_2 \otimes_A 1, \dots , x_r \otimes_A 1 \in M \otimes_A K$
	are linearly independent over $K$.
	Let $f:A^{\oplus r} \to M$ be an $A$-homomorphism 
	such that $f(e_j) = x_j$ holds,
	where $(e_j)_{j=1,2,\dots,r}$ is a free basis of $A^{\oplus r}$. 
	Let $a = \sum_{j=1}^r a_j e_j \in A^{\oplus r}$
	with $a_j \in A$.
	If $f(a) =0$ holds,
	then we have 
	$\sum_{j=1}^r a_j ( x_j \otimes_A 1)=0$.
	Thus, $a_j=0$ holds for every $j \in \{1,2,\dots,r\}$.
	Hence $f$ is injective.
	
	Next, we show that (2) implies (1).
	By tensoring with $K$, 
	the injective $A$-homomorphism $A^{\oplus r} \to M$ induces
	the injective $K$-homomorphism $K^{\oplus r} \to M \otimes_A K$.
	Thus, we have $r \le \dim_K (M \otimes_A K) = \glen_A\, M $
	by Remark~\ref{rem:GR-1}.
\end{proof}

\begin{lem}\label{lem:App-2}
	Let $A$ be a finitely generated $\Z$-algebra, 
	let $\Mm$ be a maximal ideal of $A$, and
	let $e \in \Z_{>0}$.
	Then, $A/\Mm ^e$ is a finite set.
	In particular,
	$A/\Mm$ is of positive characteristic.
\end{lem}

\begin{proof}
	Let $\alpha:\Z \to A/\Mm$ be the natural ring homomorphism.
	Then, $\alpha$ induces the injective ring homomorphism
	$\bar{\alpha}:\Z/\Ker\, \alpha \to A/\Mm$.
	Since $A/\Mm$ is a field,
	also $\Z/\Ker\, \alpha$ is a field.
	In other words,
	$\Z/\Ker\, \alpha$ is equal to $\Fp$ for some prime number $p$.
	Since $A/\Mm$ is a finitely generated $\Fp$-algebra
	and $\dim (A/\Mm) = 0$,
	we have 
	$\transdeg_{\Fp} (A/\Mm) = 0$
	by \cite[Theorem~5.6]{Mat}.
	Thus,
	$A/\Mm$ is a finite dimensional $\Fp$-vector space.
	Hence $A/\Mm$ is a finite field.
	
	
	We have an exact sequence 
	\[
	0
	\to \Mm^{e-1}/\Mm^e
	\to A/\Mm^e
	\to A/\Mm^{e-1}
	\to 0
	\]
	for every $e \in \Z_{>0}$.
	Since $\Mm^{e-1}/\Mm^{e}$ is 
	a finite dimensional $A/ \Mm$-vector space,
	we have $\# (\Mm^{e-1}/\Mm^{e}) < \infty$.
	By induction on $e$,
	we may assume that $\# (A/\Mm^{e-1}) < \infty$.
	Thus, 
	we have 
	\[
	\#(A/\Mm^{e}) = \#(\Mm^{e-1}/\Mm^{e})  \#(A/\Mm^{e-1}) < \infty.
	\]
\end{proof}

\begin{lem}\label{lem:App-3}
	Let $A$ be a noetherian integral domain,
	let $\Mm$ be a maximal ideal of $A$, 
	let $M$ be a finitely generated $A$-module, and
	let $r \in \Z_{>0}$.
	If $\#A = \infty$ and $\glen_A\,M \ge 1$,
	then 
	$\# (M \otimes_A A/\Mm^e) > r$ holds
	for sufficiently large $e \in \Z_{>0}$.
\end{lem}

\begin{proof}
	By the natural surjection 
	$M \otimes_A A/\Mm^{e+1} \to M \otimes_A A/\Mm^e$,
	it is enough to show that
	for any $r \in \Z_{>0}$, 
	there exists a positive integer $e \in \Z_{>0}$
	such that $\# (M \otimes_A A/\Mm^e) > r$.
	By $\glen_A\,M \ge 1$ and Proposition~\ref{GR-4},
	there is an injective $A$-homomorphism $\phi:A \to M$.
	Let $\hat{A}$ (resp. $\hat{M}$) be the completion of $A$ (resp. $M$) with respect to $\Mm$.
	By \cite[Theorem~8.10]{Mat},
	the natural $A$-homomorphism $\alpha: A \to \hat{A}$ is injective.
	Since $\alpha$ is flat by \cite[Theorem~8.8]{Mat},
	$\phi' 
	\coloneqq \phi \otimes_A \id_{\hat{A}}: 
	A \otimes_A \hat{A} \to M \otimes_A \hat{A}$
	is injective.
	By \cite[Theorem~8.7]{Mat},
	the natural $\hat{A}$-homomorphism
	$M \otimes_A \hat{A} \to \hat{M}$
	is an isomorphism.
	Then, we have the injective $A$-homomorphism
	\[
	A 
	\stackrel{\alpha}{\to} \hat{A} 
	\stackrel{\simeq}{\to} A \otimes_A \hat{A}
	\stackrel{\phi'}{\to} M \otimes_A \hat{A}
	\stackrel{\simeq}{\to} \hat{M}.
	\]
	Thus,
	$\# A = \infty$
	implies
	$\# \hat{M} = \infty$.
	Assume that there is a positive integer $r \in \Z_{>0}$
	such that 
	$\# (M / \Mm^e M) = \# (M \otimes_A A/\Mm^e) \le r$
	for every $e \in \Z_{>0}$.
	Then, the natural surjection 
	$M / \Mm^{e+1} M \to M / \Mm^{e} M$
	is an isomorphism for sufficiently large $e \in \Z_{>0}$.
	Thus, for such $e$, 
	$\hat{M} \simeq M / \Mm^{e} M$ holds.
	Hence we have $\# \hat{M} \le r$, which induces a contradiction.
\end{proof}

\begin{lem}\label{lem:App-4}
	Let $f:M \to N$ be an $A$-homomorphism.
	\begin{enumerate}
		\item\label{enu:App-4-1}
			Let $n$ be a positive integer.
			For each $i \in \{1,2,\dots,n\}$,
			let $N_i$ be an $A$-submodule of $N$, and
			let $M_i \coloneqq f^{-1}(N_i)$.
			If $\bigcup_{i=1}^n M_i \neq M$, 
			then $\bigcup_{i=1}^n N_i \neq N$.
		\item\label{enu:App-4-2}
			Let $n$ be a positive integer.
			For each $i \in \{1,2,\dots,n\}$,
			let $M_i$ be an $A$-submodule of $M$, and
			let $N_i \coloneqq f(M_i)$.
			If $f$ is surjective and 
			$\bigcup_{i=1}^n N_i \neq N$,
			then $\bigcup_{i=1}^n M_i \neq M$.
	\end{enumerate}
\end{lem}

\begin{proof}
	(\ref{enu:App-4-1})
	Let $x \in M$.
	Then, $f(x) \in \bigcup_{i=1}^n N_i$ implies
	$x \in \bigcup_{i=1}^n M_i$.
	Thus, we have that
	if $\bigcup_{i=1}^n M_i \subsetneq M$
	then $\bigcup_{i=1}^n N_i \subsetneq f(M) \subseteq N$.
	
	(\ref{enu:App-4-2})
	Let $x \in M$.	
	Then, $x \in \bigcup_{i=1}^n M_i$ implies
	$f(x) \in \bigcup_{i=1}^n N_i$.
	Thus, we have that
	if $\bigcup_{i=1}^n N_i \subsetneq f(M) = N$
	then $\bigcup_{i=1}^n M_i \subsetneq M$.
\end{proof}

\begin{lem}\label{lem:App-main}
	Let $R$ be a noetherian ring, and
	let $M$ be a finitely generated $R$-module.
	Let $n$ be a positive integer.
	For each $i \in \{1,2,\dots,n\}$,
	let $\Mp_i \in \Spec\,R$,
	let $R_i \coloneqq R / \Mp_i$, and
	let $M_i$ be an $R$-submodule of $M$.
	For each $i \in \{1,2,\dots,n\}$,
	assume the following conditions:  
	\begin{itemize}
		\item
			$\Mp_i (M/M_i) = 0$,
			in particular,
			$M/M_i$ is an $R_i$-module;
		\item
			if $\dim R_i = 0$, $\glen_{R_i} (M/M_i) > \log_2 n$;
		\item
			if $\dim R_i \ge 1$, $\glen_{R_i} (M/M_i) \ge 1$.
	\end{itemize}
	Then, $\bigcup_{i=1}^n M_i \neq M$ holds.
\end{lem}

\begin{proof}
	First, we show the case 
	where $R$ is a finitely generated $\Z$-algebra.
	\begin{claim}\label{claim:App-main-fgZ}
		If $R$ is a finitely generated $\Z$-algebra,
		the assertion holds.
	\end{claim}
	\begin{proof}[Proof of Claim~\ref{claim:App-main-fgZ}]
	For each $i \in \{1,2,\dots,n\}$,
	pick a maximal ideal $\Mm_i$ of $R$ 
	such that $\Mp_i \subseteq \Mm_i$.
	Let $e$ be a sufficiently large integer.
	Let $\Mq \coloneqq \prod_{i=1}^n \Mm_i ^e$,
	let $M' \coloneqq M/\Mq M$, and
	let $M'_i \coloneqq f(M_i)$, 
	where $f:M \to M'$ be the natural surjection.
	By Lemma~\ref{lem:App-4}~(\ref{enu:App-4-2}),
	it is enough to show that
	$\bigcup_{i=1}^n M'_i \neq M'$.
	
	We now show $\# M' < \infty$.
	After permuting the indices,
	we may assume that 
	there is an integer $n' \in \{1,2,\dots,n\}$
	such that
	$\sequence{\Mm}{,}{n'}$ are pairwise distinct 
	and 
	$\{\sequence{\Mm}{,}{n'}\} = \{\sequence{\Mm}{,}{n}\}$ holds.
	It holds that
	$\Mq = \prod_{j=1}^{n'} \Mm_j ^{e_j}$
	for some $\sequence{e}{,}{n'} \in \Z_{>0}$.
	Given $j, j' \in \{1,2,\dots,n'\}$ with $j \neq j'$,
	$\sqrt{\Mm_j ^{e_j} + \Mm_{j'} ^{e_{j'}}} 
	\supseteq \Mm_j + \Mm_{j'}
	= R$ holds,
	and hence we have 
	$\Mm_j ^{e_j} + \Mm_{j'} ^{e_{j'}} = R$.
	By the Chinese remainder theorem, 
	we have 
	$R/\Mq \simeq 
	\prod_{j=1} ^{n'} R/\Mm_j ^{e_j}$.
	By Lemma~\ref{lem:App-2},
	we have 
	$\# (R/\Mm_j ^{e_j}) < \infty$
	for all $j \in \{1,2,\dots,n'\}$.
	Thus, we also have 
	$\#(R/\Mq) < \infty$.
	Since $M'$ is finitely generated $R/\Mq$-module, 
	we have $\# M' < \infty$.
	
	We now show $\#(M'/M'_i) > n$
	for any $i \in \{1,2,\dots,n\}$.
	Fix $i \in \{1,2,\dots,n\}$.
	If $\dim R_i = 0$,
	then we have 
	\[
	M_i 
	\supseteq \Mp_i M 
	= \Mm_i M
	\supseteq \Mq M,
	\]
	since $\Mp_i = \Mm_i$.
	Thus, we have 
	$M' / M'_i \simeq M/ (\Mq M + M_i) = M/ M_i$.
	Note that 
	$R_i$ is a field,
	$M/M_i$ is a finite dimensional $R_i$-vector space,
	and $\dim_{R_i} (M/M_i) = \glen_{R_i} (M/M_i)$.
	Hence we have
	\[
	\#(M' / M'_i) 
	= \#(M/ M_i) 
	= (\# R_i) ^{\glen_{R_i}\,(M/M_i)}
	> 2^{\log_2 n}
	= n.
	\]
	If $\dim R_i \ge 1$,
	we have a surjection
	\[
	M' / M'_i 
	\stackrel{\simeq}{\to} M/ (\Mq M + M_i)
	\to M/ (\Mm_i^e M + M_i)
	\stackrel{\simeq}{\to} M/M_i \otimes_{R_i} R_i/\bar{\Mm}_i^e,
	\]
	where $\bar{\Mm}_i = \Mm_i/\Mp_i$.
	Thus, we have 
	$ \#(M' / M'_i) \ge \#(M/M_i \otimes_{R_i} R_i/\bar{\Mm}_i^e)$.
	We show $\#R_i =  \infty$.
	If $R_i$ is of characteristic $0$,
	it is clear.
	On the other hand,
	if $R_i$ is of characteristic $p>0$,
	we have $\transdeg_{\Fp}\; K(R_i) = \dim R_i \ge 1$,
	where the first equality follows from \cite[Theorem~5.6]{Mat}.
	Thus there exists a transcendental element in $R_i$ over $\Fp$,
	which implies $\#R_i =\infty$.
	By $\#R_i = \infty$ and Lemma~\ref{lem:App-3},
	we have 
	$\#(M/M_i \otimes_{R_i} R_i/\bar{\Mm}_i^e) > n$.
	Thus, we have 
	$\#(M' / M'_i) > n$.
	
	For each $i \in \{1,2,\dots,n\}$,
	$\#(M' / M'_i) > n$ implies
	\[
	\#M'_i = \frac{\#M'}{\#(M' / M'_i)}  < \frac{\#M'}{n}.
	\]
	Thus, we have
	\[
	\#\left(\bigcup_{i=1}^n M'_i \right) 
	\le \sum_{i=1}^n \#M'_i 
	< n \cdot \frac{\#M'}{n} 
	= \#M' 
	< \infty,
	\]
	which implies
	$\bigcup_{i=1}^n M'_i \neq M'$.
	\end{proof}
	Then, we show the general case.
	Let $g:B \to R$ be 
	a ring homomorphism from a polynomial ring 
	$B \coloneqq \Z[\sequence{t}{,}{l}]$
	such that
	for every $i \in \{1,2,\dots,n\}$,
	the following condition holds:
	\begin{itemize}
		\item
			if $R_i$ is of characteristic $p>0$ and $\dim R_i \ge 1$, 
			then 
			for some $j \in \{1,2,\dots,l\}$,
			the image of $t_j$ under the induced ring homomorphism 
			$B \stackrel{g}{\to} R \to R_i$ is a transcendental element over $\Fp$.
	\end{itemize}
	Let $\alpha:R^{\oplus N} \to M$ be a surjective $R$-homomorphism
	with $N \in \Z_{>0}$,
	and 
	let $h:F \to M$ be the composite $B$-homomorphism
	\[
	F
	\coloneqq B^{\oplus N} 
	\stackrel{g^{\oplus N}}{\longrightarrow} R^{\oplus N}
	\stackrel{\alpha}{\to} M.
	\]
	For each $i \in \{1,2,\dots,n\}$,
	let $F_i \coloneqq h^{-1} (M_i)$, and
	let $B_i \coloneqq B/g^{-1}(\Mp_i)$.
	By Lemma~\ref{lem:App-4}~(\ref{enu:App-4-1}),
	it is enough to show that
	$\bigcup_{i=1}^n F_i \neq F$.
	Fix $i \in \{1,2,\dots,n\}$.
	$g^{-1}(\Mp_i)$ is a prime ideal of $B$.
	Let $\bar{h}_i: F/F_i \to M/M_i$ be the natural injective $B$-homomorphism 
	induced by the $B$-homomorphism $F \stackrel{h}{\to} M \to M/M_i$.
	We have 
	\[
	g^{-1}(\Mp_i) (F/F_i) = \bar{h}_i^{-1}(\Mp_i (M/M_i)) = \bar{h}_i^{-1}(0) = 0.
	\]
	Note that 
	the ring homomorphism $B \stackrel{g}{\to} R \to R_i$ 
	induces the injective ring homomorphism
	$B_i \to R_i$
	and that
	$F/F_i$ is a $B_i$-module.
	Let $\phi': F \times R \to M$ be the $B$-baranced map 
	such that $\phi'(f,r) = rh(f)$.
	By the universal property of the tensor product, 
	we obtain the $B$-homomorphism $\phi:F \otimes_B R \to M$.
	Let $\bar{\phi}'_i : (F/F_i) \times R_i \to M/M_i$ be the $B$-baranced map 
	such that $\bar{\phi}'_i(f,r) = r \bar{h}_i(f)$.
	By the universal property of the tensor product,
	we obtain the $B$-homomorphism $\bar{\phi}_i:(F/F_i) \otimes_B R_i \to M/M_i$.
	Then the following diagram is commutative:
	\[
	\begin{CD}
		F \otimes_B R 					@>\phi>>					M 		\\
		@VVV 																	@VVV 	\\
		(F/F_i) \otimes_B R_i 		@>\bar{\phi}_i>> 	M/M_i,	
	\end{CD}
	\]
	where the vertical arrows are the natural surjective $B$-homomorphisms.
	Since $\phi$ is equal to the composition 
	\[
	F \otimes_B R 
	= B^{\oplus N} \otimes_B R 
	\stackrel{\simeq}{\to} R^{\oplus N}
	\stackrel{\alpha}{\to} M,
	\]
	$\phi$ is surjective.
	Thus, also $\bar{\phi}_i$ is surjective.
	Hence we have 
	\begin{align} \label{ali:App-main-1}
	\glen_{R_i}\, (M/M_i)
	&\le \glen_{R_i} ((F/F_i) \otimes_B R_i) \notag \\
	&= \glen_{R_i} ((F/F_i) \otimes_{B_i} R_i)  \\
	&= \glen_{B_i}\, (F/F_i), \notag
	\end{align}
	where the last equality follows from Lemma~\ref{lem:App-1}.
	In particular, we have $\glen_{B_i}\, (F/F_i) \ge 1$.
	\begin{claim}\label{claim:App-2}
		If $\dim B_i = 0$, then $\dim R_i = 0$.
	\end{claim}
	\begin{proof}[Proof of Claim~\ref{claim:App-2}]
	Assume that $\dim B_i = 0$ and $\dim R_i \ge 1$.
	By Lemma~\ref{lem:App-2},
	$B_i$ is of characteristic $p>0$,
	and hence also $R_i$ is of characteristic $p$.
	We consider the following commutative diagram:
	\[
	\begin{CD}
		@. 			 			B 			@>g>> 		R 			\\
		@. 						@VVV 					@VVV 	\\
		\Fp 		@>>> 		B_i 		@>>> 		R_i. 		
	\end{CD}
	\]
	By the definition of $g$,
	the image of $t_j$ under  
	the ring homomorphism $B \stackrel{g}{\to} R \to R_i$ 
	is a transcendental element over $\Fp$
	for some $j \in \{1,2,\dots,l\}$.
	Since the ring homomorphism $B \stackrel{g}{\to} R \to R_i$ 
	is equal to the ring homomorphism $B \to B_i \to R_i$,
	also the image of $t_j$ under 
	the ring homomorphism $B \to B_i$ 
	is transcendental over $\Fp$.
	Since $B_i$ is a finitely generated $\Fp$-algebra,
	$\dim B_i$ is equal to $\transdeg_{\Fp}\, K(B_i)$
	(cf. \cite[Theorem~5.6]{Mat}).
	Thus, we have $\dim B_i \ge 1$.
	This is a contradiction.
	\end{proof}
	By Claim~\ref{claim:App-2} and inequality~(\ref{ali:App-main-1}),
	we have that
	if $\dim B_i = 0$,
	then 
	\[
	\log_2 n < \glen_{R_i}\, (M/M_i) \le \glen_{B_i}\, (F/F_i)
	\]
	holds.
	By Claim~\ref{claim:App-main-fgZ},
	we have $\bigcup_{i=1}^n F_i \neq F$,
	and hence the assertion holds.
\end{proof}

\begin{thm}\label{thm:App-main}
	Let $f:X \to S$ be a projective morphism of noetherian schemes,
	with $X$ reduced and $S$ affine.
	Let $L$ be an invertible sheaf on $X$, and
	let $A$ be an $f$-ample invertible sheaf on $X$.
	Then, 
	for sufficiently large $m \in \Z_{>0}$, 
	there is a section
	$s \in H^0(X,L^{\otimes m} \otimes_{\MO_X} A^{-1})$
	such that 
	for any irreducible component $V$ of $X$,
	if $L|_V$ is $f|_V$-big, then $\Supp(V) \nsubseteq \Supp(Z(s))$ holds.
\end{thm}

\begin{proof}
	If $X = \emptyset$, the assertion clearly holds.
	Thus we may assume that $X \neq \emptyset$.
	Let $\sequence{Y}{,}{l}$ be all the connected components of $X$
	equipped with the reduced scheme structure.
	Then, by the Mayer--Vietoris sequence,
	we have 
	\[
	H^0(X,L^{\otimes m} \otimes_{\MO_X} A^{-1})
	\simeq \bigoplus_{i=1}^l 
	 			H^0(Y_i,(L|_{Y_i})^{\otimes m} \otimes_{\MO_{Y_i}} (A|_{Y_i})^{-1})
	\]
	for any $m \in \Z_{>0}$.
	Thus, it is enough to show that
	the assertion holds for any connected component of $X$.
	Hence we may assume that $X$ is connected.
	
	Let $\sequence{X}{,}{n}$ be all the irreducible components of $X$.
	Let 
	\[
	I \coloneqq \{i \in \{1,2,\dots,n\} \mid \text{$L|_{X_i}$ is $f|_{X_i}$-big} \}.
	\]
	For each $i \in I$,
	let $X'_i \coloneqq \bigcup_{j \neq i} X_j$ equipped with the reduced scheme structure, 
	let $T_i$ be the scheme-theoretic intersection of $X_i$ and $X'_i$,
	and let $d_i \coloneqq \dim(X_i/f(X_i))$.
	Note that $d_i \ge 0$ for each $i \in I$.
	We write 
	$S = \Spec\,R$ with a noetherian ring $R$.
	For each $i \in I$,
	we write 
	$f(X_i) = \Spec\,R_i$ with a noetherian integral domain $R_i$,
	and 
	$R_i = R/\Mp_i$ with $\Mp_i \in \Spec\,R$.
	Let $H_m \coloneqq L^{\otimes m} \otimes_{\MO_X} A^{-1}$
	for $m \in \Z_{>0}$.
	Fix $i \in I$.
	%
	By Proposition~\ref{WB-lower bound},
	there is a positive real number $C'_i \in \R _{>0}$
	such that
	for sufficiently large $m>0$, 
	\[
	 \glen_{R_i}\, H^0({X_i},H_m|_{X_i}) \geq C'_i m^{d_i}.
	\]
	holds.
	By Proposition~\ref{RD-<},
	we have $\dim(T_i/f(X_i)) \le d_i-1$.
	By Lemma~\ref{WB-upper bound}, 
	there is a positive real number $C''_i \in \R _{>0}$ 
	such that 
	for sufficiently large $m>0$,
	\[
	 \glen_{R_i}\, H^0(T_i,H_m|_{T_i}) 
	 \leq C''_i m^{\dim(T_i/f(X_i))} 
	 \leq C''_i m^{d_i-1}
	\]
	holds.
	By the Mayer--Vietoris sequence,
	we have the exact sequence
	\[
	0 
	\to H^0(X,H_m) 
	\to H^0({X_i},H_m|_{X_i}) \oplus H^0({X'_i},H_m|_{X'_i}) 
	\to H^0({T_i},H_m|_{T_i}). 
	\]
	Let $\phi_{i,m}:H^0(X,H_m) \to H^0(X_i,H_m|_{X_i})$
	and $\psi_{i,m}:H^0(X_i,H_m|_{X_i}) \to H^0(T_i,H_m|_{T_i})$
	be the natural homomorphisms.
	By the above exact sequence, 
	we have $\Im\,\phi_{i,m} \supseteq \Ker\,\psi_{i,m}$.
	Hence for sufficiently large $m \in \Z_{>0}$, 
	we have
	\begin{align*}
				&\glen_{R_i}(H^0(X,H_m)/\Ker\,\phi_{i,m}) \\
	= 			&\glen_{R_i}(\Im\,\phi_{i,m}) \\
	\ge 		&\glen_{R_i}(\Ker\,\psi_{i,m}) \\
	\ge 		&\glen_{R_i}\,H^0({X_i},H_m|_{X_i}) - \glen_{R_i}\,H^0({T_i},H_m|_{T_i}) \\
	\ge 		&C'_i m^{d_i} - C''_i m^{d_i-1} \\
	\ge 		&C_i m^{d_i},
	\end{align*}
	where $C_i \coloneqq C'_i / 2$.
	Since $d_i \ge 0$, 
	$\glen_{R_i}(H^0(X,H_m)/\Ker\,\phi_{i,m}) \ge 1$ holds for sufficiently large $m \in \Z_{>0}$.
	If $\dim X_i = 0$,
	then we have $X = X_i$
	since $X$ is connected and $X_i$ is an irreducible component of $X$.
	Hence, the assertion clearly holds
	if $\dim X_i = 0$.
	Thus, we may assume that 
	$\dim X_i \ge 1$.
	If $\dim R_i = 0$,
	we have $\dim(X_i / f(X_i)) = \dim X_i$ 
	by the definition of $\dim(X_i / f(X_i))$,
	and hence we have $d_i = \dim(X_i / f(X_i)) \ge 1$.
	Then, for sufficiently large $m \in \Z_{>0}$,
	we have
	\[
	\glen_{R_i}(H^0(X,H_m)/\Ker\,\phi_{i,m})
	\ge C_i m^{d_i}
	> \log_2 n.
	\]
	
	By Lemma~\ref{lem:App-main},
	we have
	$\bigcup_{i \in I} \Ker\,\phi_{i,m} \neq H^0(X,H_m)$
	for sufficiently large $m \in \Z_{>0}$.
	In other words,
	for sufficiently large $m \in \Z_{>0}$,
	there is a section $s \in H^0(X,L^{\otimes m} \otimes_{\MO_X} A^{-1})$
	such that
	$s|_{X_i} \neq 0$ holds
	for any $i \in I$.
	By Proposition~\ref{prop:ZL-3} and Proposition~\ref{prop:ZL-5},
	the assertion holds.
\end{proof}

\begin{cor}\label{App-Kodaira's lemma}
	Let $f:X \to S$ be a projective morphism of noetherian schemes,
	with $X$ reduced and $S$ affine.
	Let $L$ be an invertible sheaf on $X$, and
	let $A$ be an $f$-ample invertible sheaf on $X$.
	If $L|_V$ is $f|_V$-big
	for any irreducible component $V$ of $X$,
	then for sufficiently large $m \in \Z_{>0}$,
	there is an effective Cartier divisor $D$ on $X$ 
	such that $L^{\otimes m} \simeq A \otimes_{\MO_X} \MO_X(D)$.
\end{cor}

\begin{proof}
	By Theorem~\ref{thm:App-main},
	for sufficiently large $m \in \Z_{>0}$,
	there is a section $s \in H^0(X,L^{\otimes m} \otimes_{\MO_X} A^{-1})$
	such that 
	$\Supp(V) \nsubseteq \Supp(Z(s))$ holds 
	for any irreducible component $V$ of $X$.
	By Definition-Proposition~\ref{prop:ZL-effective divisor},
	$Z(s)$ defines an effective Cartier divisor $D$ on $X$.
	Thus, we have 
	$L^{\otimes m} \otimes_{\MO_X} A^{-1} \simeq \MO_X(D)$,
	i.e.,
	$L^{\otimes m}  \simeq  A \otimes_{\MO_X} \MO_X(D)$.
\end{proof}

\newcommand{\etalchar}[1]{$^{#1}$}


\end{document}